\documentclass[12pt]{amsart}

\def\eps{{\varepsilon}}

\def\mes{{\rm mes}}

\def\Card{{\rm Card}}

\def\Prob{{\mathbb{P}}}

\def\bbH{\mathbb{H}}

\def\bbW{\mathbb{W}}

\def\naturals{\mathbb{N}}

\def\Tor{\mathbb{T}}

\def\reals{\mathbb{R}}

\def\integers{\mathbb{Z}}

\def\RmII{{I\!\!I}}

\def\RmIII{{I\!\!I\!\!I}}

\def\bC{\mathbf{C}}

\def\bG{\mathbf{G}}

\def\bm{\mathbf{m}}

\def\bp{\mathbf{p}}

\def\bt{\mathbf{t}}
\def\bv{\mathbf{v}}

\def\bx{\mathbf{x}}

\def\bkappa{\boldsymbol{\kappa}}

\def\brho{\boldsymbol{\rho}}

\def\bsigma{\boldsymbol{\sigma}}

\def\brC{{\bar C}}

\def\brH{{\bar H}}

\def\bra{{\bar a}}

\def\brc{{\bar c}}

\def\brx{{\bar x}}

\def\bry{{\bar y}}

\def\breta{{\bar \eta}}

\def\cA{\mathcal{A}}

\def\cB{\mathcal{B}}

\def\cC{\mathcal{C}}

\def\cD{\mathcal{D}}

\def\cI{\mathcal{I}}

\def\cF{\mathcal{F}}

\def\cH{\mathcal{H}}

\def\cE{\mathcal{E}}

\def\cL{\mathcal{L}}

\def\cM{\mathcal{M}}

\def\cN{\mathcal{N}}

\def\cO{\mathcal{O}}

\def\cS{\mathcal{S}}

\def\cT{\mathcal{T}}

\def\cW{\mathcal{W}}

\def\fA{\mathfrak{A}}

\def\fc{\mathfrak{c}}

\def\fg{\mathfrak{g}}

\def\fh{\mathfrak{h}}

\def\fm{\mathfrak{m}}

\def\fp{\mathfrak{p}}

\def\hC{{\hat C}}

\def\hX{{\hat X}}

\def\tH{{\tilde H}}

\def\tx{{\tilde x}}

\makeatother

\usepackage{xcolor}

\usepackage[margin=3cm]{geometry}

\usepackage[normalem]{ulem}
\usepackage{amsthm, bbm}
\usepackage{amsmath, amssymb}
\usepackage{mathrsfs}
\usepackage{hyperref}
\usepackage{graphicx, psfrag, epstopdf}

\theoremstyle{definition}
\newtheorem{definition}{Definition}[section]
\newtheorem{theorem}[definition]{Theorem}
\newtheorem{lemma}[definition]{Lemma}
\newtheorem{corollary}[definition]{Corollary}
\newtheorem{proposition}[definition]{Proposition}
\newtheorem{problem}[definition]{Problem}

\newtheorem{remark}[definition]{Remark}
\newtheorem{example}[definition]{Example}
\numberwithin{equation}{section}

\definecolor{OliveGreen}{rgb}{0,0.6,0}

\def\DS{\displaystyle}
\def\R{\mathbb{R}}
\def\Z{\mathbb{Z}}
\def\T{\mathbb{T}}
\def\N{\mathbb{N}}
\def\a{{\bf a}}
\def\Ab{{\mathbb{A}}}

\author{D. Dolgopyat, C. Dong, A. Kanigowski, and P. N\'andori}
\title[Flexibility of statistical properties for smooth systems with the CLT]
{Flexibility of statistical properties for smooth systems satisfying the central limit theorem}

\begin{document}

\begin{abstract}
In this paper we  exhibit new classes of smooth systems which satisfy the Central Limit Theorem (CLT) and have (at least) one of the following properties:
\begin{itemize}
\item zero entropy;
\item weak but not strong mixing;  
\item (polynomially) mixing but not $K$; 
\item  $K$ but not Bernoulli;
\item non Bernoulli and mixing at arbitrary fast polynomial rate.
\end{itemize}
 We also give an example
of a system satisfying the CLT where the normalizing sequence is regularly varying with index $1$.
\end{abstract}

\maketitle

\tableofcontents

\part{Results}
\section{Partially chaotic systems}
\label{ScPartCahos}
An important discovery made in the last century is that deterministic systems can exhibit chaotic behavior.
Currently there are many examples of systems enjoying a full array of chaotic properties 
which follow from either uniform hyperbolicity or non-uniform hyperbolicity, in case there is 
a control on the region
where hyperbolicity is weak
\cite{BDV05, Bow75, CM06, You98}. 
Systems which satisfy only some of the above properties are less
understood. In fact, it is desirable to have more examples of such systems in order to understand 
the full range of possible behaviors of partially chaotic systems.

The Central Limit Theorem (CLT) is 
a hallmark of chaotic behavior. 
There is a vast literature on the topic. In particular there are numerous methods of establishing
CLT, including the method of moments (cumulants) \cite{BG, Ch95}, 
spectral method \cite{Gou15}, the martingale method \cite{Gor69, HH80, KKM18}
(the list of references here is by no means exhaustive, we just provide a sample of papers
which could be used for introducing non-experts to the corresponding techniques).
However, the above methods require strong mixing properties of the system.
As a result, they apply to systems which have strong statistical properties
including Bernoulli property and summable decay of correlations.
The only example going beyond strongly chaotic framework is the product 
of an Anosov\footnote{The methods of \cite{CC13} apply to more general systems in the first factor,
however, they seem insufficient to produce the examples described in Theorems 
\ref{ThCLT+K}--\ref{ThWeakMix}.}
 diffeomorphism (called diffeo in the sequel) and a Diophantine rotation,
which is shown in \cite{CC13} to satisfy the CLT (see also \cite{K98, PW10} or Corollary \ref{CrUE-CLT} below).

Thus the knowledge on possible ergodic behaviors of smooth systems satisfying CLT is very restricted. The main goal of this paper is to provide new classes of systems satisfying 
CLT with interesting ergodic properties.

More precisely,
let $F$ be a $C^r$ diffeomorphism of a smooth orientable manifold $M$ 
preserving a measure $\zeta$ which
is absolutely continuous with respect to volume. 
Let $C^r_0(M)$ denote the space of $C^r$ functions of zero mean,
 i.e. satisfying $\zeta(A)=0.$  For an observable $A$, consider the ergodic sums
$$ A_N(x)=\sum_{n=0}^{N-1} A(F^n x). $$ 

\begin{definition}
\label{DefCLT} {We say that $F$ satisfies the}
{\em Central Limit Theorem (CLT)} if there is a sequence $a_n$ such that
 for each $A \in  C^r_0(M)$, $\frac{A_n}{a_n}$ converges in law as $n\to\infty$
to normal random variable with zero mean and variance $\sigma^2(A)$ 
(such normal random variable will be denoted $\cN(0, \sigma^2(A))$ in the sequel)
and moreover that 
$\sigma^2(\cdot)$
is not identically equal to zero on $C^r_0(M).$
We say that $F$ satisfies the {\em classical CLT}
if one can take $a_n=\sqrt{n}.$ 
\end{definition}

\begin{definition}
\label{def:mixing}
We say that $F$ is mixing at rate $\psi$ if for any
$A_1, A_2\in C^r_0(M)$ 
the correlation function $\rho_n(A_1, A_2)=\zeta(A_1\cdot(A_2\circ F^n))$ satisfies
\begin{equation}
\label{EqPM}
\left|\rho_n(A_1, A_2)\right|\leq  \|A_1\|_{C^r} \|A_2\|_{C^r} \psi(n).
\end{equation}
In case $\psi(n) = C n^{- \delta}$ for some $C, \delta>0,$
we say that $F$ is {\em polynomially mixing}. If  
$\psi(n) = C e^{-\delta n}$ for some $C, \delta>0,$
we say that $F$ is {\em exponentially mixing}.
\end{definition}

The above definitions can be extended to flows in a straightforward way. 

We say that a system is $K$ if it has no non-trivial zero entropy factor; it is Bernoulli if it is isomorphic to a Bernoulli shift. We are now ready to state our main results.
\begin{theorem}
\label{ThCLT+K}
For each $m\in \naturals$ there exists an analytic diffeomorphism $F_m$  which is mixing at rate $n^{-m}$
but is not Bernoulli.
Moreover, $F_m$ is $K$ and satisfies the classical CLT.
\end{theorem}

To the best of our knowledge,
the first part of the theorem provides the first example of a system which has summable correlations
but is not Bernoulli. The second (``moreover") part answers a question that we heard from multiple sources, first time from 
J-P.~Thouvenot.

We also show that the CLT does not imply positive entropy:

\begin{theorem}
\label{ThZE-CLT}
(a) There exists an analytic flow of zero entropy which satisfies the CLT with normalization 
$a_T=T/\ln^{1/4} T.$

(b) For each $r\in \naturals$ there is a manifold $M_r$ and a $C^r$ diffeo $F_r$ on $M_r$ of zero entropy which satisfies the
classical CLT. 
\end{theorem}
We note that in all previous results on the CLT, the normalization was regularly varying with index 
$\frac{1}{2}.$ \footnote{\label{FtNLnN} CLT with normalization $\sqrt{n\ln n}$ appears for
expanding and hyperbolic maps with neutral fixed points \cite{Gou04, Br19}, as well as in several hyperbolic billiards
\cite{BCD11, BG06, SzV07}. In a followup paper we will show it also appears for generalized 
$T, T^{-1}$ transformations with hyperbolic base and two parameter exponentially mixing flows in the
fiber.}

 We also give examples of weakly mixing but not mixing as well as polynomially mixing but not K systems satisfying the CLT.
 
\begin{theorem} 
\label{ThWeakMix}
(a) There exists a weakly mixing but not mixing flow, which satisfies the classical CLT. 

(b) There exists a polynomially mixing flow, which is not $K$ and satisfies the classical CLT. 
\end{theorem}

All the systems in Theorems \ref{ThCLT+K}--\ref{ThWeakMix} belong to the class of generalized $T,T^{-1}$ transformations which are described in Section \ref{sec:TT-1} below. In order to construct our examples we need to extend significantly the existing methods for proving both the CLT and the non Bernoulli
property of these maps. {In fact, the main difficulty in Theorems \ref{ThZE-CLT} and \ref{ThWeakMix} is to establish the CLT while other properties are rather straightforward. On the other hand, the main difficulty in Theorem \ref{ThCLT+K} is to show non Bernoullicity. More details} on the general framework for proving the CLT for generalized $T, T^{-1}$ transformations is presented in Section \ref{ScCLTResults}, while the precise results pertaining to the non Bernoullicity
are described in Section \ref{ScNBResults}.

\section{Generalized $T, T^{-1}$ systems}
\label{sec:TT-1}
Generalized $T, T^{-1}$ transformation is a classical subject 
(see \cite{Hal, Mei74, We72} and reference therein for the early work on this topic) 
{with a rich range of applications in} probability and ergodic theory.
In fact, generalized $T, T^{-1}$ transformations were used to exhibit examples 
of systems with unusual limit laws \cite{KS, CC17},
central limit theorem with non standard normalization \cite{Bolt}, K but 
non Bernoulli systems
in abstract \cite{Kal82} and smooth setting in various dimensions \cite{Kat80, Rud, KRHV},
very weak Bernoulli but not weak Bernoulli partitions \cite{dHKSS},
slowly mixing systems \cite{dHS97, LB, DDKN}, systems with multiple Gibbs measures \cite{EL04, MN79}.
Here, we exhibit further ergodic and statistical properties of these systems.

To define smooth $T, T^{-1}$ transformations,
let $X, Y$ be compact manifolds, $f:X\to X$ be a smooth map preserving 
a measure $\mu$ and $G_t: Y\to Y$ be a $d$ parameter
flow on $Y$ preserving a measure $\nu.$ 
Throughout this work, we assume that
$G_t$ is exponentially mixing of all orders (see \eqref{eq:memixing} for a precise 
definition).
Let $\tau: X\to\reals^d$ be a smooth map.
We study the following map $F: X\times Y\to X\times Y$
\begin{equation}
\label{TTInvDef}
F(x,y)=(f(x), G_{\tau(x)}y). 
\end{equation}
Note that $F$ preserves the measure $\zeta=\mu\times \nu$ and that 
$$ F^N(x,y)=(f^N x, G_{\tau_N(x)} y)\quad\text{where}\quad 
\tau_N(x)=\sum_{n=0}^{N-1} \tau(f^n x). $$

We also consider continuous  $T,T^{-1}$ transformations. Namely let $h_t$ be a flow on $X$ preserving $\mu$.
Set
\begin{equation}\label{eq:Tcont}
F_T(x,y)=(h_T(x),G_{\tau_T(x)}y)
\quad\text{where}\quad \tau_T(x)=\int_0^T\tau(h_t x)dt.
\end{equation}

In the literature, generalized $(T,T^{-1})$ transformations are sometimes called  Kalikow systems. If $d\geq 2$, we call them {\it higher rank Kalikow systems}.

\section{Central Limit Theorem for $T, T^{-1}$ transformations}
\label{ScCLTResults}

Here we present sufficient results for $T, T^{-1}$ transformations defined by \eqref{TTInvDef} (and \eqref{eq:Tcont}) to satisfy the CLT. The results of this section will be proven in Part \ref{PtCLT}.

\subsection{Continuous actions in the fiber}
\label{SSResCont}
Let $f$
and $G_t$ be as in Section \ref{sec:TT-1}.
We assume that $G_t$ enjoys exponential mixing of all orders.  
In the case $d\geq 2$ our main example is the following:
 $Y=SL_{d+1}(\reals)/\Gamma$, 
 $G_t: Y\to Y$ is the  Cartan action 
on $Y$ (see Example \ref{eg1} for more details), and $\nu$ is the Haar measure. 
For $d=1$ there are more examples, see e.g. the discussion in \cite{DDKN}. Given a  H\"older function $H: X\times Y\to \reals$, let
$$  H_N{(x,y)}=\sum_{n=0}^{N-1} H(F^n(x,y)). $$
We want to study the distribution of $H_N$ when the initial condition $(x,y)$ is distributed according to
$\zeta.$

The definition and the results below are stated for discrete time. However they can be directly translated to continuous time. The discrete versions are used in Theorems  \ref{ThCLT+K} and \ref{ThZE-CLT}(b), whereas the continuous versions are used in Theorem \ref{ThWeakMix}.

\begin{definition}
We say that $\tau$ satisfies {\em polynomial large deviation bounds} 
if there exists $\kappa>0$ such that for each $\eps>0$ there exists $C$
such that for any $N\in\naturals $,
\begin{equation}
\label{EqLD-Poly}
 \mu\left(
\left\|
\frac{\tau_N}{N}- \mu(\tau)
\right\|\geq \eps\right)\leq \frac{C}{N^\kappa}.
\end{equation} 
\end{definition}

\begin{theorem}
\label{ThCLT3}Suppose that the base map satisfies the  following: there exist $r$ such that for each $A\in C^r(M)$ with $\mu(A)=0$,
there is a number $\sigma^2(A)\geq 0$ such that $\frac{A_N}{\sqrt{N}}\to \cN(0, \sigma^2(A))$
as $N\to\infty.$\footnote{{Note that in contrast with Definition \ref{DefCLT} we do not require $\sigma^2(A)$ 
to be generically non-zero.}} Suppose furthermore that there is some
$\eps >0$ and $C$ so that for every $N$,
\begin{equation}
\label{eq:I1}
\mu\left(\| \tau_N \|<\log^{1+ \eps}N\right)<\frac{C}{N^{5}}.
\end{equation}
Then there is $\Sigma^2$ such that
$\DS \frac{ H_N}{\sqrt{N}}$ converges as $N\to \infty$ to 
the normal distribution with mean $\zeta(H)$  and variance $\Sigma^2.$
\end{theorem}

In particular we have the following 
corollary.

\begin{corollary}
\label{CrUE-CLT}Suppose $\tau$ satisfies \eqref{eq:I1} and for 
all
smooth mean zero functions $A$, $\DS A_N/\sqrt{N}$ converges in law to zero as $N\to\infty$.
 Then there is $\Sigma^2$ such that
$\DS \frac{H_N}{\sqrt{N}}$ converges as $N\to \infty$ to 
the normal distribution with
mean $\zeta(H)$  and variance $\Sigma^2.$ 
\end{corollary}

\begin{remark}
\label{RmPosAv}
We remark that if
$\tau$ satisfies polynomial large deviations bounds with $\mu(\tau)\neq 0$ and 
$\kappa \geq 5$, then \eqref{eq:I1} holds.
This property is sometimes more convenient to check.
 In particular, \eqref{eq:I1} is satisfied if $\tau$ is strictly 
positive. In fact, it is sufficient that there is a constant $a$ such that $m(\tau)>a$ for each 
$f$ invariant measure $m.$ The later condition is convenient for systems which have a small number
of invariant measures,  such as flows on surfaces considered in \S \ref{SSSWM}.
\end{remark}

\subsection{Discrete actions} 
The problem discussed in \S \ref{SSResCont} also makes sense when $G$ is an action of $\integers^d$
and $\tau:X\to\integers^d$ is a map satisfying 
\eqref{eq:I1}. In \S \ref{SSResCont} we restricted our
attention to continuous actions, since our motivation is to construct smooth systems with
exotic properties, however all the results presented above remain valid for $\integers^d$-actions.
The proof requires minor modifications since the approach presented below requires only the
smoothness with respect to $y$, but not with respect to $x.$
Therefore, we leave both formulations and proofs to the readers.

\subsection{Previous results}
\label{subsec:RWRS}

The first results about $T, T^{-1}$ transformations
 pertain to so called {\em random walks in random scenery}\footnote{We refer to \cite{Pe20} for a
 survey of limit theorems for random walks in random scenery.}
 . In this model
we are given a sequence $\{\xi_z\}_{z\in \integers^d}$ of i.i.d. random variables.
Let $\tau_n$ be a simple random walk on $\integers^d$ independent of $\xi$s. We are interested
in  $\DS S_N=\sum_{n=0}^{N-1} \xi_{\tau_n}.$ This model could be put in the present framework as
follows. Let $X$ be a set of sequences  $\{v_n\}_{n\in \integers},$ 
where $v_n\in \{\pm e_1, \pm e_{2}, \dots \pm e_d\}$ where $e_j$ are basis vectors in $\integers^d,$
$\mu$ is the Bernoulli measure with 
$\mu(v_n=\pm e_j)=\frac{1}{2d}$ for all $n \in \integers$ and for all $j\in 1, \dots, d,$
$Y$ is the space of sequences $\{\xi_z\}_{z\in\integers^d}$, $\nu$ is the product with marginals induced by 
$\xi$, $f$ and $G$ are shifts and $\tau(\{v\})=v_0.$ For random walks in random scenery,
 the CLT is due to \cite{Bolt}.
In the context of dynamical systems, Theorem \ref{ThCLT3} 
was proven
in \cite{DDKN} assuming that $f$ enjoys multiple exponential mixing. 
The case $d=1$ which leads to a non Gaussian limit was analyzed in \cite{LB} using the techniques
of stochastic analysis.
In the present paper we follow a
method of \cite{Bolt} which seem more flexible and allows a larger class of base systems.
In the dynamical setting the strategy of \cite{Bolt} amounts to regarding $F$ as a 
{\em Random Dynamical System (RDS)} on $Y$ driven by $f.$ We first prove a quenched 
CLT for typical realization of the noise $x$ and then show that the parameters of
the CLT are almost surely
constant. Limit Theorems for RDS were studied in a number of papers (see e.g. \cite{K98}). 
The novelty of the present setting is that instead of requiring hyperbolicity
in the fibers we assume just mixing. This requires a different  approach. We relate the problem of CLT for the $T,T^{-1}$ transformation to fluctuations of ergodic sums of the skewing function. In particular, in several interesting cases
we are able to show that typically the ergodic sums of the skewing function $\tau$ are large on the logarithmic scale but small on the scale $N^{1/2}$. This new approach has the potential to be applicable to more general context.

\subsection{Examples}
\label{SSEx}
 Here we describe several applications of our results on the CLT including
systems substantiating Theorems \ref{ThCLT+K}, \ref{ThZE-CLT} and \ref{ThWeakMix}.

 \subsubsection{Anosov base}
Let $f$ be an Anosov diffeo preserving  a Gibbs measure $\mu$.

\begin{theorem}
\label{ThAnosovCLT}
Suppose that either

(i) $\mu(\tau)\neq 0$ or

(ii) $d\geq 3.$

Then $F$ satisfies the classical CLT.
\end{theorem}
This theorem was previously proven in \cite[Corllary 5.2]{DDKN}.
Here we show that Theorem~\ref{ThAnosovCLT} fits in the framework of the present paper.
Also, in Part \ref{PtKalikow} we shall show that the map $F$ from Theorem 
\ref{ThAnosovCLT}(ii) is not Bernoulli,
so this result will serve as an example of Theorem \ref{ThCLT+K}.
We note that part (i) of Theorem~\ref{ThAnosovCLT} directly follows from Theorem \ref{ThCLT3} since
in this case we have exponential large deviations (\cite{K90}). The derivation of part (ii) 
using the methods of the present paper will be given in 
\S \ref{SSAnCLT}.

\subsubsection{Theorem \ref{ThZE-CLT} (a)}
Let $d=1$ and let $Q$ be a hyperbolic surface of 
constant negative curvature of arbitrary genus $p \geq 1$.
Let
$h_t$ be the (stable) horocycle flow on the
unit tangent bundle $X=SQ$, that is,
$h_t$ is moving $x \in X$ at unit speed along its
stable horocycle
\begin{equation}
\label{def:hor}
\cH(x) = 
\{ \tx \in X: 
\lim_{t \to \infty}
d( {\bf G}_t(x),  {\bf G}_t(\tx)) = 0
\}
\end{equation}
where $\bf G_t$ is the geodesic flow on $X$. Let $\tau:X\to \reals$ be a smooth mean-zero cocycle defined as follows: let $\gamma_1, \dots, \gamma_{2p}$ be the basis in homology of $Q.$ Choose $i\in \{1, \dots 2p\}$
and let $\lambda$ be a
closed form on $Q$ such that 
\begin{equation}
\label{DualBase}
\int_{\gamma_j} \lambda=\delta_{ij} 
\end{equation}
where $\delta$ is the Kronecker symbol.
Take 
\begin{equation}
\label{HoroTau}
\tau(q,v)=\lambda(q) (v^*)
\end{equation}
 where 
$v^*$ is a unit vector obtained from $v$ by the $90$ degree rotation.
 We assume that the $\mathbb{R}$ action $(G_t, Y,\nu)$ is exponentially mixing of all orders and 
 consider the system (see \eqref{eq:Tcont})
$$
F_T(x,y)=(h_T(x),G_{\tau_T(x)}y).
$$
We have
\begin{equation}
\label{Wind}
 \tau_T(x)=\int_{\fh(x, T)} \lambda 
\end{equation}
where $\fh(x,T)$ is the projection of the horocyle starting from
$x$ and
of length $T$, to $Q$.

Let $H: X\times Y\to \reals$ be a smooth observable. 

The next result is a more precise version
of Theorem \ref{ThZE-CLT}(a).

\begin{theorem}
\label{ThHoro}
There exists $\sigma^2\geq 0$ such that 
$\frac{(\ln T)^{1/4} \; H_T}{T}$ converges as $T\to\infty$ to 
the normal distribution with zero mean and variance $\sigma^2.$
\end{theorem}

Assuming Theorem \ref{ThHoro}, we can complete the proof of 
Theorem \ref{ThZE-CLT}(a) by showing that the limiting variance is not identically
zero (which  will be done in
Section \ref{ScVar}) 
and that $F_T$ has zero entropy. The latter statement 
 is a consequence of the following lemma (which is formulated for maps, 
so we apply it for the time $1$ map $F_1$ to conclude that
the continuous $T,T^{-1}$ system has zero entropy).

\begin{lemma}
\label{LmBFZE}
Let $F$ be a generalized $T, T^{-1}$ transformation such that $h_\mu(f)=0$ and
$\mu(\tau)=0.$ Then $h_\zeta(F)=0.$
\end{lemma}

\begin{proof}
By the classical Abramov-Rokhlin entropy addition formula 
(\cite{AR}),
$$h_\zeta(F)=h_\mu(f)+\sum_i \max\{\chi_i(\mu(\tau)), 0\}$$
where $\chi_i:\R^d\to \reals$ are Lyapunov functionals of $G_t$
(we refer to Section \ref{sec:cartan} for
the background on this notion).
In our case the first term vanishes since the base has zero entropy, and
the second term vanishes since $\chi_i(\mu(\tau))=\chi_i(\mathbf{0})=0.$
\end{proof}


\begin{remark}
The proof of Theorem \ref{ThHoro} given in Section \ref{ScHoro} applies to a slightly more general situation.
Namely, using the ideas from \cite{DN-Flows} one can consider the case of $\tau:X\to \reals^d$ 
where each component of $\tau$ is of the form $\lambda(q) (v^*)$ where $\lambda$ is a closed form
(not necessarily taking integer values on the basis loops). We note that by the results of \cite{FF03}
every function which has non-zero components only in the discrete series representation is homologous to 
a function of the form \eqref{HoroTau}. On the other hand \cite{FF03} also shows that for general smooth functions on $X$ the behavior of ergodic integrals is very different. Therefore the results
of Section \ref{ScHoro} do not apply to the general observables on $X.$ 
Similarly, the example of Theorem \ref{ThZE-CLT}(b) also requires a careful choice of 
the skewing function.
\end{remark}

\subsubsection{Theorem \ref{ThZE-CLT} (b)}
In this section we will construct, for any fixed $r\in \mathbb{N}$, a $C^r$ 
zero entropy system for which the classical central limit theorem holds. Let $\bm\in \mathbb{N}$, $\alpha\in \mathbb{T}^\bm$. We say that 
$\alpha\in \mathbb{D}(\bkappa)$ if there exists $D>0$ such that for every $k\in \mathbb{Z}^\bm$, 
$$|\langle k, \alpha\rangle|\geq D |k|^{-\bkappa}.$$
Recall that $\mathbb{D}(\bkappa)$ is non-empty if $\bkappa\geq \bm$ and it has full measure if $\bkappa>\bm.$
For $d\in \mathbb{N}$ let $(G_t, M,\nu)$ be a $\reals^d$ action which is exponentially mixing of all orders.


The main tool for constructing the example for Theorem \ref{ThZE-CLT}(b) 
is a fine control of ergodic averages of the translation by $\alpha$. 
Namely, in Section \ref{ScToral} we prove:

\begin{proposition}\label{prop:dioph}
For every $\bkappa/2<r<\bm$, there is a $d\in \mathbb{N}$ such that for every 
$\alpha\in \mathbb{D}(\bkappa)$, we have:

\begin{enumerate}
\item[D1.]  for every $\phi \in \bbH^r( \mathbb{T}^\bm,\reals)$ of zero mean, 
$\|\phi_n\|_2=o(n^{1/2})$; (Here  $\bbH^r(\mathbb{T}^\bm,\reals)$ denotes the Sobolev space of order $r$)

\item[D2.]  there is a function $\tau:=\tau^{(\alpha)}\in C^{r}(\mathbb{T}^\bm,\reals^d)$ such that 
$\mu(\tau) = \mathbf{0}$ and 
$$
\mu\Big(\{x\in \mathbb{T}^\bm\;:\; \|\tau_n(x) \|<\log^2n\}\Big)=o(n^{-5}).
$$
\end{enumerate}
\end{proposition}

\begin{proof}[Proof of Theorem \ref{ThZE-CLT} (b)] 
Let $F$ be the $T,T^{-1}$
transformation with $f$ being the translation of the 
$\mathbf{m}$-torus by $\alpha$ and $\tau$ as provided
by Proposition \ref{prop:dioph}(D2).
$F$ has zero entropy by Lemma \ref{LmBFZE}.
The CLT follows from Corollary \ref{CrUE-CLT}. Namely, property 
\eqref{eq:I1} follows from D2, and
$A_N/\sqrt{N}\to 0$ in law by D1, since $C^r\subset \bbH^r.$ 
The fact that the limiting variance is not identically
zero follows from
Section \ref{ScVar}.
\end{proof}


\subsubsection{ Theorem \ref{ThWeakMix}}
\label{SSSWM}

Let $\alpha\in \mathbb{T}$ be an irrational number. Let $f:\mathbb{T}\to \reals_+$ be a function which is $C^3$ on $\Tor\setminus\{0\}$, $\int f d Leb = 1$ and $f$ satisfies
\begin{equation}\label{eq:asy2}
\lim_{\theta\to 0^+}\frac{f''(\theta)}{h''(\theta)}=A\text{ and }
\lim_{\theta\to 1^-}\frac{f''(\theta)}{h''(1-\theta)}=B,
\end{equation}
where $A^2+B^2\neq 0$ and the function $h$ is specified below.
 We consider the {\em special flow} over $R_\alpha \theta=\theta+\alpha$ and under $f$.
This flow acts on $X=\{(\theta,s): \theta\in \T, 0\leq s<f(\theta) \}$ by
$$
T_t^f(\theta,s)=(\theta+N(\theta,s,t)\alpha,s+t-f_{N(\theta,s,t)}(\theta)),
$$
where $N(\theta,s,t)$ is the unique number such that 
$f_{N(\theta,s,t)}(\theta)\leq s+t<f_{N(\theta,s,t)+1}(\theta)$ (where 
$f_n(\theta) = \sum_{k=0}^{n-1} f(\theta + k \alpha)$).
Such special flows arise as representations of a certain class of smooth flows on surfaces:
\begin{enumerate}
\item if $h(\theta)=\log \theta$ and  $A\neq B$, then the flow $T_t^f$ represents the restriction to the ergodic component of a smooth flow $(\varphi_t)$ on $(\mathbb{T}^2,\mu)$ with one fixed point and one saddle loop. 
Here $\mu$ is given by $p(\cdot)vol$, for some smooth function $p$. Such flows are mixing for a.e. irrational rotation \cite{KSin}.
\item if $h(\theta)=\log \theta$ and $A=B$, then for every irrational $\alpha$ and any surface $M$ with genus $\geq 2$, the flow represents a certain ergodic smooth flow $(\varphi_t)$ on $(M,\mu)$ (see e.g. \cite{Kochergin}, \cite[Proposition 2]{Fr-Lem}). Here $\mu$ is locally given by $p(\cdot)vol$, for a smooth function $p$. 
Such flows are not mixing, \cite{Kochergin}, but weakly mixing for a.e. $\alpha$, \cite{Fr-Lem}.
\item if   $h(\theta)= \theta^{-\gamma}$,  then for some values of $\gamma<1$
 the flow $T_t^f$ represents an ergodic smooth flow $(\varphi_t)$ on $\mathbb{T}^2$ (as shown in \cite{Kochergin2} this is the case, in partcular if $\gamma=1/3$). 
 Moreover by \cite{Kochergin2} $(\varphi_t)$ is mixing for every $\alpha$ and by \cite{Fayad2} if $\gamma\leq 2/5$, then the flow is {\em polynomially} mixing for a full measure set of $\alpha$. In what follows we will always assume that $\gamma\leq 2/5$ 
 (although the proof can be applied for $\gamma<1/2$ with minor changes).
\end{enumerate}


We consider the continuous flow $F_T$ given by (see \eqref{eq:Tcont})
$
F_T(x,y)=(\varphi_T(x), G_{\tau_T}(y)),
$
where $\varphi_t$ is as in (1) or (2) or (3) above,  $(G_t,Y,\nu)$ is an exponentially 
mixing $\mathbb{R}$-flow and $\tau$ is positive.
For $\brH \in C^3(X)$, let 
$$
\brH_U(\theta, s)=\int_{0}^U\brH(T^f_u(\theta, s))du.
$$

Let $\DS \cC^3=\{\brH\in C^3(X):  \bp(\brH):=\lim_{s\to\infty} \brH(0,s)\text{ exists}\}.$
Note that functions on $X$ correspond to
functions on the surface which are $C^3$ with
$\bp(\cdot)$ being the value of the function at the fixed point of the flow. The next result is proven in Section \ref{ScSurfaceFlows}.

\begin{proposition}\label{prop:deverg}There exists $\epsilon>0$ such that for a.e. $\alpha$ and 
for every
 $\bar{H}\in  \cC^{3}$
$$
\mu\left(\Big\{x\in X\;:\; \Big|\bar{H}_T(x)-T\mu(\bar{H})\Big|={\rm O}(T^{1/2-\epsilon})\Big\}\right)=1-o(1),\text{ as } T\to \infty.
$$
\end{proposition}

\begin{proof}[Proof of Theorem \ref{ThWeakMix}]
 For part (a) let $\varphi_t$ be as in (2). 
To see that $F_T$ is weakly mixing we note that
$F$ is relatively mixing in the fibers (see \cite{DDKN}), so any eigenfunction should be constant
in the fibers and whence constant since $G_T$ is
weakly mixing. 

For part (b) let $\varphi_t$ be as in (3). 
Note that in both cases the CLT follows from continuous versions of Corollary \ref{CrUE-CLT} and Remark \ref{RmPosAv}
(the fact that the limiting variance is not identically
zero follows from
Section \ref{ScVar}).

Taking $\varphi_t$ as in (1) gives different examples of mixing (on an ergodic component) but not $K$ systems satisfying the CLT.
\end{proof}

\section{Non Bernoulicity of $T, T^{-1}$ transformations}
\label{ScNBResults}
Here we show that $T, T^{-1}$ transformations with Anosov base and exponentially mixing fiber 
are non Bernoulli. The proofs are given in part \ref{PtKalikow}.

Let $f:X\to X$ be a diffeomorphism preserving a topologically mixing basic hyperbolic set 
$\Lambda$ and let $\mu$ be a Gibbs measure with H\"older potential on $\Lambda.$
Let $d$ be a positive
integer. Assume one of the following
\begin{itemize}
\item[{\bf a1.}] $Y$ is a nilmanifold, $\nu$ be a Haar measure,
$\Ab=\Z^d$ and  $G_t$ is an $\Ab$
 action on $Y$ by hyperbolic affine maps.
\item[{\bf a2.}] $\Ab\in\{\Z^d,\R^d\}$, $Y$ is a quotient of a semisimple Lie group $H$ by a co-compact irreducible lattice and $G$ 
is a partially hyperbolic action on $Y$ such that the restriction of $\alpha$ to the center space is identity (see \eqref{cen:id}).
\end{itemize}
Let
$\tau:X\to\Ab$ be a H\"older mean zero cocycle which is not homologous to a cocycle taking value
in a proper subgroup of $\Ab$
(this assumption does not lead to a loss of generality
since homologous cocycles give rise to conjugated maps, and if $\tau$ takes value in a proper subgroup
we can consider the smaller action from the beginning).
We consider the skew product \eqref{TTInvDef}.

Our main result is:
\begin{theorem}\label{thm:main0}
If $G$ is as in {\bf a1.} or {\bf a2.} then $F$ is not Bernoulli. 
\end{theorem}

 Let us give examples of systems satisfying the assumptions {\bf a1.} and {\bf a2.} respectively.

\begin{example}[Cartan action on $\T^n$]
\label{eg1}
Let $n\ge 3$. A $\Z^{n-1}$-action by hyperbolic automorphisms of the $n$-dimensional torus is called {\em Cartan action}. One can construct concrete examples by considering embedding of algebraic number fields to $\R$. For more details, we refer to \cite{KKS}. Multiple exponential mixing for such actions is proven in much more general setting, see \cite{GS}.
\end{example}

\begin{example}[Weyl Chamber flow on $SL(n,\R)/\Gamma$]
\label{eg2}
Let $n\ge 3$, and $\Gamma$ be a uniform lattice in $SL(n,\R)$. Let $D_+$ be the group of diagonal elements in $SL(n,\R)$ with positive elements. It is easy to see that $D_+$ is isomorphic to $\R^{n-1}$. The group $D_+$ acts on $SL(n,\R)/\Gamma$ by left translation. Thus we obtain a $\R^{n-1}$ action, which is called {\em Weyl Chamber flow}. 
A crucial property of Weyl chamber flow is (multiple) exponential mixing. Exponential mixing is proven by using 
matrix coefficients in \cite{KS1}, and multiple exponential mixing is established in \cite{BEG}.
\end{example}

In case when $f$ is an Anosov map, $\Ab=\R^d$, $\tau$ is smooth, and $d\geq 3$, the map $F$ discussed above satisfies the assumptions of Theorem 
\ref{ThCLT+K}. Indeed, the $K$ property for $F$ follows from
Corollary 2 in \cite{Gui89}, the CLT follows from Theorem \ref{ThCLT3} 
(or from \cite[Theorem 5.1]{DDKN})  and mixing with rate $n^{-d/2}$ 
follows from 
\cite[Theorem 4.7(a)]{DDKN}.

Markov partitions allow to construct a measurable isomorphism between the 
hyperbolic basic sets with a Gibbs measure and a subshift 
of finite type (SFT) with a Gibbs measure (\cite{Bow75}). Let $\sigma:(\Sigma_A,\mu)\to(\Sigma_A,\mu)$ be a topologically transitive SFT with a Gibbs measure $\mu$. Theorem \ref{thm:main0} immediately follows from

\begin{theorem}\label{thm:main}
Let $G$ be as in {\bf a1.} or {\bf a2.}, $\tau$ be a H\"older mean zero cocycle on $\Sigma_A$ and 
$$
F(\omega,y)=(\sigma \omega,G_{\tau(\omega)}y).
$$
Then $F$ is not Bernoulli. 
\end{theorem}

One of the main steps in the proof of Theorem \ref{thm:main} is to show that relative atoms (on the fiber) of the past partition are points (see Proposition \ref{prop:atoms}). If  $G$ is a $\Z^d$ Bernoulli shift 
with $d= 1,2$, then the assertion of Proposition \ref{prop:atoms} is still true. This is a consequence of the fact that the corresponding $\Z^d$ random walk is recurrent. In particular in this setting one can easily adapt our proof to cover the examples
  by Kalikow \cite{Kal82}, Rudolph \cite{Rud} and den Hollander--Steif, \cite{dHS97} (in slightly wider generality as they assume $\sigma$ is the full shift). 
On the other hand if $d\geq 3$, then the $\Z^d$ random walk is not recurrent and the assertion of Proposition \ref{prop:atoms} does not hold. In fact,  the main result in \cite{dHS97} says that in this case $F$ is Bernoulli (if $\sigma$ is the full shift).
It is known that symbolic and smooth actions of rank $\geq 2$
are quite different. For example, a higher rank smooth action has zero
entropy \cite{OW} but there is an abundance of symbolic actions with positive entropy. Theorem \ref{thm:main} is another
 manifestation on the difference between smooth and symbolic actions of higher rank.
\\

Our approach is motivated by \cite{Kal82, Rud}. In particular, the statement of the key Proposition 
\ref{lem:tpr} is similar to the corresponding statements of \cite{Kal82, Rud}. However, its proof in our case is different, since the other authors rely on fine properties of the ergodic sums of the 
cocycle $\tau$ while our approach uses exponential mixing in the fiber.
We note that in dimension $d\geq 3$ we can have the same cocycle $\tau$ but different fiber dynamics,
namely, a random walk, and get a Bernoulli system, so using fiber dynamics is essential.
Another important ingredient to our approach is the use of Bowen-Hamming distance 
(see Proposition \ref{prop:VWB}) 
which allows us to handle continuous higher rank actions in the fiber, and so it plays a crucial role in 
constructing the example of Theorem \ref{ThCLT+K}. We also emphasize that the systems considered in \cite{Kal82, Rud} were shown by the authors not to be {\em loosely Bernoulli}. We  believe that our methods would work also to show non loose Bernoullicity at a cost of rather technical combinatorial considerations as one needs to consider the $\bar{f}$ metric instead of the Hamming metric. To keep the presentation relatively simple and since our goal was to establish smooth $K$ but non Bernoulli examples satisfying CLT, we restrict our attention to only deal with non Bernoullicity.\\

 We note that the assumption that $\tau$ has zero mean in the above theorems  
 is essential. Indeed, if $\tau$ has non-zero mean, then by  
 \cite[Theorem 4.1(a)]{DDKN}, $F$ is exponentially mixing, and then one can show using the
argument of \cite{Ka-Bern} that $F$ is Bernoulli. The details will be given in a separate
paper \cite{DKRH}.


\section{Flexibility of statistical properties}
\subsection{Overview}
Here we put the results of Sections \ref{ScCLTResults} and \ref{ScNBResults} into a more general framework.

There is a vast literature on statistical properties of dynamical systems.  A survey  \cite{Sin00} lists
the following hierarchy of statistical properties for dynamical systems preserving 
a smooth measure (the properties
marked with * are not on the list in \cite{Sin00} but 
we added them to obtain a more
complete list).

(1) {\bf (Erg)} Ergodicity;
(2*) {\bf (WM)} Weak Mixing
(3) {\bf (M)} Mixing;
(4*) {\bf (PE)} Positive entropy;
(5) {\bf (K)} K property;
(6) {\bf (B)} Bernulli property;
(7*) {\bf (LD)} Large deviations; 
(8) {\bf (CLT)} Central Limit Theorem\footnote{\cite{Sin00} refers to classical CLT, but since the time
it was written several CLTs with non classical normalization has been proven, cf. footnote
\ref{FtNLnN}.};
(9*) {\bf (PM)} Polynomial mixing;
(10) {\bf (EM)} Exponential mixing.

Properties (1)--(6) are qualitative. They make sense for any measure preserving dynamical system.
Properties (7)--(10) are quantitative. They require smooth structure but provide quantitative estimates.
Namely let $F$ be a $C^r$ diffeomorphism of a smooth orientable manifold $M$ with a fixed volume form preserving a measure $\mu$ which
is absolutely continuous with respect to volume.
Recall that a formal definition of {\bf (CLT)} 
{\bf (PM)} and {\bf (EM)}
were given in Section \ref{ScPartCahos}.
By {\bf (LD)} we mean 
{\em exponential large deviations}, that is  {
for each $\eps>0$
there exists $\delta>0$ and $C$ such that 
for every $N$ and}
for any function $A\in C^r(M)$ of zero mean
$$  \mu(x: |A_N(x)|\geq \eps N)\leq C \|A\|_{C^r} e^{-\delta N}, $$
where 
$\DS  A_N(x)=\sum_{n=0}^{N-1} A(F^n x) $ 
are the ergodic sums.

The same definitions apply to flows with obvious modifications. While 
properties
on the bottom of the list are often more difficult to establish especially in the context of nonuniformly 
hyperbolic systems discussed in \cite{Sin00}
it is not
true that property $(j)$ on this list implies all the properties ($i$) with $i\leq j.$ This leads to the following

\begin{problem}
\label{PrFlex}
  Study logical independence of the properties from the list above. That is, given two disjoint subsets
  $\cA_1, \cA_2\subset \{1, \dots, 10\}$ determine if there exists a smooth map preserving a smooth probability measure
  which has all properties from $\cA_1$ and does not have any properties from $\cA_2.$
\end{problem}  
If $\cA_1$ contains some properties from the bottom of our list while $\cA_2$ contains 
some 
properties from the top, then 
an affirmative answer to Problem \ref{PrFlex} provides
exotic examples exhibiting a new type of stochastic behavior
in deterministic systems. On the other hand finding new implications among properties (1)--(10)
would also constitute an important advance
since it would tell us that once we checked some properties from our list, some additional properties
are obtained as a free bonus.

Of course, the solution of
Problem \ref{PrFlex} in all the cases where $|\cA_1|+|\cA_2|=10$ would immediately imply the solution for
all the cases where $\cA_1\cup \cA_2$ is a proper subset of our list.  However, 
 the cases where $\cA_1\cup\cA_2$ is small, are of a higher practical interest, since any non-trivial
 implication between the properties in $\cA_1\cup \cA_2$ lead to simpler theorems.
We note that all cases with $\cA_1=\emptyset$ can be realized with taking $F=id$ and
all cases with $\cA_2=\emptyset$ can be realized by Anosov diffeomorphisms, so the problem is 
non-trivial only if both $\cA_1$ and $\cA_2$ are non-empty.
Thus the simplest non-trivial case of the problem is 
the case
where both $\cA_1$ and $\cA_2$ consist of a single element.
The known results are summarized
in the table below. Here Y in cell $(i, j)$ means that the property in 
row $i$ implies the property in the column $j.$ $(k)$ in cell $(i,j)$ means that a diffeo number $(k)$ 
on the list below has property $(i)$ but not property $(j).$

The examples in the table below are the following (the papers cited in the list contain results needed to verify some properties
in the table):

(1) irrational rotation; (2) almost Anosov flows studied in \cite{Br19}
\footnote{
In the table we use the fact that the maps with neutral periodic points do not
satisfy LD. Indeed for such maps, if $x$ is $\eps=1/T^k$-close to a neutral periodic orbit $\gamma$,
with $k$ large enough, then $A_T(x)$ is close to $T \int_\gamma A$ which may be far from $T\mu(A)$
if $\mu(A)\neq \int_\gamma A.$ More generally for maps which admit Young tower
with polynomial tail, large deviations have polynomial rather than exponential probabilities.
We refer the reader to \cite{Mel09, GM14, DN-Ren} for discussion of precise large deviation bounds
in that setting.};
(3) horocycle flow (\cite{BF14}); 
(4) Anosov diffeo $\times$ identity; 
(5) maps from Theorem \ref{ThZE-CLT};
(6) skew products on $\Tor^2\times \Tor^2$ of the form $(Ax, y+\alpha \tau(x))$ where $A$ is linear Anosov 
map, $\alpha$ is Liouvillian and $\tau$ is not a coboundary \cite{D02}; 
(7) Anosov diffeo$\times$Diophantine rotation (see \cite{K98, CC13} and Corollary~\ref{CrUE-CLT}).

\medskip

\begin{center}
\begin{tabular}{|c|c|c|c|c|c|c|c|c|}
\hline
& {\bf Erg} & {\bf WM/M} & {\bf PE} & {\bf K/B} & {\bf LD} & {\bf CLT} & {\bf PM} & {\bf EM} \cr
{\bf Erg} & $\clubsuit$ & (1) & (1) & (1)  &(2) & (1) & (1) & (1) \cr
{\bf WM/M} & Y  & $\clubsuit$ & (3) & (3)  & (2) & (6) & (6) & (6) \cr
{\bf PE} & (4) & (4)  & $\clubsuit$ & (4) &  (4) & (4) & (4) & (4)  \cr
{\bf K/B} & Y & Y &  Y & $\clubsuit$  & (2) & (6) & (6) & (6) \cr
{\bf LD} & Y & (1) & (1) & (1)  &   $\clubsuit$ & (1) & (1) & (1)  \cr
{\bf CLT} & Y & (7) & (5) & (7)  &  (2) & $\clubsuit$ & (7) & (7)   \cr
{\bf PM} & Y & Y & (3) & (3)  & (2) &  (3) & $\clubsuit$ & (3)   \cr
{\bf EM} & Y & Y & ?? & ?? & ??  &  ??& Y & $\clubsuit$    \cr
\hline
\end{tabular} 
\end{center}
\medskip

We combined {\bf (WM)} and {\bf (M)} 
(as well as {\bf (K)} and {\bf (B)})
together since the same counter examples work for both properties. 
It is well known that weak mixing does not imply mixing (see Section \ref{ScSurfaceFlows}) and
that $K$ does not imply Bernoulli (see Section \ref{ScNBResults}).

The positive implications in the top left $4\times 4$ corner are standard and can be found in most 
textbooks on ergodic theory. It is also clear that
Exponential Mixing $\Rightarrow$ Polynomial Mixing $\Rightarrow$ Mixing and that both
CLT and Large Deviations imply the weak law of large numbers which in turn entails ergodicity.

There are 4 cells  with the question mark, all of them concentrated in 
{\bf (EM)} row. 
This problem is addressed in an ongoing work, \cite{DKRH}, in which the authors show that if a $C^2$ volume preserving diffeomorphism is exponentially mixing, then it is Bernoulli.

The remaining two cells in {\bf (EM)} row seem hard.  For example,
it is known (\cite{Ch95}, see also \cite{BG}) that the classical CLT follows from {\em multiple
exponential mixing,} that is,  the CLT holds if for each $m$ 
\begin{equation}
\label{eq:memixing}
\left|\int \left(\prod_{j=1}^m A_j(f^{n_j} x) \right) d\mu(x) -\prod_{j=1}^m \mu(A_j)\right|
\leq C_m \prod_{j=1}^m \|A_j\|_{C^r} \; e^{-\delta_m \min_{i\neq j} |n_i-n_j|} .
\end{equation}
Therefore the question if exponential mixing implies CLT is related to the following 
\begin{problem}
Does exponential mixing imply multiple exponential mixing?
\end{problem}
\noindent which a quantitative version of a famous open problem of Rokhlin. The above problem is also interesting in a more general context whether mixing with a certain rate implies higher order mixing with the same rate.

In the construction used to prove Theorem \ref{ThZE-CLT}(b), $\dim(M_r)$ grows linearly with $r$
which leads to the following natural question:
\begin{problem}
Construct a $C^\infty$ diffeomorphism with zero entropy satisfying the classical CLT.
\end{problem}

The table also shows that {\bf (PM)} does not imply any qualitative properties stronger than 
mixing. However in the counter example listed in the table the mixing is quite slow in the sense that
$\psi(n) = Cn^{- \delta}$ in \eqref{EqPM} with $\delta < 1$. This leads to the following problem.

\begin{problem}
Given $m\in \naturals$ construct a diffeomorphism which is mixing at rate $n^{-m}$ and

(a) is not $K$;

(b) has zero entropy;

(c) does not satisfy the CLT.
\end{problem}

Positive implications in our table suggest the following more tractable version of Problem \ref{PrFlex}.
Let ${\bf (NE)}$, $\overline{(\bf E)}$, $\overline{\bf (WM)}$,
 $\overline{\bf (M)}$, $\overline{\bf (PM)}$, ${\bf (EM)}$ denote the systems
 which are respectively non-ergodic, ergodic, weakly mixing, mixing, polynomially mixing, or
 exponentially mixing, but do not have any stronger properties on this list.
 Likewise let $\bf (ZE)$, $\overline{\bf (PE)}$, $\overline{\bf (K)}$, $\bf (B)$
 denote the systems which are respectively zero entropy, positive entropy, $K$ or Bernoulli,
 but do not have any stronger properties on our list. Then Problem \ref{PrFlex} is equivalent to
 \medskip
 
\noindent{\bf Problem \ref{PrFlex}*}
Given $P_1\in \{{\bf (NE)}, \overline{(\bf E)}, \overline{\bf (WM)},
 \overline{\bf (M)}, \overline{\bf (PM)}, {\bf (EM)}\}$,\\
 $P_2\in \{{\bf (ZE)}, \overline{\bf (PE)}, \overline{\bf (K)}, {\bf (B)}\}$,
 $P_3\in \{ {\bf (CLT)}, {\bf non (CLT)}\},$
 $P_4\in \{ {\bf (LD)}, {\bf non (LD)}\}$
 does there exist a smooth dynamical system with properties $P_1, P_2, P_3, P_4$?
  
Theorems \ref{ThCLT+K}, \ref{ThZE-CLT}, and \ref{ThWeakMix} provide several new examples
related to this problem.
Namely, we construct exotic systems which 
satisfy CLT and in some cases are not Bernoulli. 
For this reason we provide below a discussion Problem \ref{PrFlex}
in the case where $|\cA_1|+|\cA_2|=3$ and either CLT$\in \cA_1$ (\S \ref{SSCLT-Flex})
or
B$\in\cA_2$ (\S \ref{SSFlex-Bern}).

\subsection{CLT and flexibility}
\label{SSCLT-Flex}
Here we consider Problem \ref{PrFlex} with $|\cA_1|=2,$ $|\cA_2|=1$ and
CLT$\in \cA_1.$
The table below
lists in cell $(i, j)$ a map which has both property (i) and satisfies CLT but does not have property $j.$
Clearly the question makes sense only if we have an example of a system which has property (i) but not property (j).

\medskip
\begin{center}
\begin{tabular}{|c|c|c|c|c|c|c|c|}
\hline
& {\bf WM} & {\bf M} & {\bf PE} & {\bf K} & {\bf B} & {\bf LD} & {\bf PM} \cr
{\bf WM} & $\clubsuit$ & (9) & (10) & (10) & (10) & (2)  & (11) \cr
{\bf M} & $\clubsuit$  & $\clubsuit$ & (10) & (10) & (10) & (2)  & (11)  \cr
{\bf PE} & (7) & (7)  & $\clubsuit$ & (7) & (7) & (2) &  (7)  \cr
{\bf K} & $\clubsuit$ & $\clubsuit$ & $\clubsuit$ & $\clubsuit$ & (8) & (2) & ?? \cr
{\bf B} & $\clubsuit$ &$\clubsuit$ & $\clubsuit$ & $\clubsuit$ & $\clubsuit$ & (2)  & ??  \cr
{\bf LD} & $\clubsuit$ & (7) & ?? & (7) & (7) &   $\clubsuit$  & (7)  \cr
{\bf PM} & $\clubsuit$ & $\clubsuit$ & (10) & (10) & (10) & (2) &   $\clubsuit$  \cr   
\hline
\end{tabular} 
\end{center}
\medskip

Here (2) and (7) refer to the diffeomorphisms from the previous table, while  (8), (9), (10), and (11) and refer to the maps
from Theorems \ref{ThCLT+K}, \ref{ThWeakMix}(a), (b) and \ref{ThZE-CLT}(a).
To see that the example of Theorem \ref{ThZE-CLT}(a) is not polynomially mixing we
note that for polynomially mixing systems the growth of ergodic integrals can not be regularly
varying with index one. Namely
(see e.g. \cite[\S 8.1]{DDKN}),
for polynomially mixing systems there exists $\delta>0$ such that the
ergodic averages of smooth functions $H$ satisfy 
$\DS \lim_{T\to\infty} \frac{H_T}{T^{1-\delta}}=0$ almost surely, and hence, in law. 

The last table leads to the following questions.

\begin{problem}
Construct an example of K (or even Bernoulli) diffeomorphism which satisfies the CLT but is not polynomially mixing.
\end{problem}

\begin{problem}
Construct an example of a zero entropy map which enjoys both the CLT and the large deviations.
\end{problem}

\subsection{Flexibility and Bernoullicity}
\label{SSFlex-Bern}
Here we consider the special case of Problem \ref{PrFlex} when $|\cA_1|=2$ and
$\cA_2=$\{{\bf B}\}.  In view of \cite{DKRH}
 we assume that {\bf EM}$\not\in \cA_1.$
We may also assume that {\bf CLT}$\not\in \cA_1,$
otherwise we are in the setting of \S \ref{SSCLT-Flex}.
We note that the map of Theorem \ref{ThCLT+K} have all remaining statistical properties except,
possibly, {\bf (LD)} while the horocycle flow enjoys all those properties except being
$K.$ Thus the only remaining question in this case is

\begin{problem}
Find a system which is $K$ and satisfies the large deviation property but is not Bernoulli.
\end{problem}

\subsection{Related questions}
The questions presented below are not special cases of Problem \ref{PrFlex} but they are
of a similar spirit.

\begin{problem}
\label{PrCLTAnyManifold}
Let $M$ a compact manifold of dimension at least two. Does there exists a 
$C^\infty$ diffeomorphism of $M$ preserving a smooth measure satisfying a Central Limit Theorem?
\end{problem}

Currently it is known that any compact manifold of dimension at least two admits an ergodic
diffeomorphism of zero entropy \cite{AK70},
a Bernoulli diffeomorphism \cite{BFK81}, and, moreover, a nonuniformly hyperbolic
diffeomorphism \cite{DP02}. We note that a recent preprint \cite{PSS20}
constructs 
area preserving diffeomorphisms on any surface of class $C^{1+\beta}$ (with $\beta$ small) which satisfy
both {\bf (CLT)} and {\bf (LD)}. It seems likely that similar constructions could be made in higher 
dimensions, however, the method of \cite{PSS20} requires low regularity to have degenerate 
saddles where a typical orbit does not spent too much time, and so the methods do not work 
in higher smoothness such as $C^2.$
We also note that \cite{BD87} shows that for any 
aperiodic dynamical system there exists
some measurable observable satisfying the CLT\footnote{One can also ask which limit distributions
can appear in the limit theorems in the context of measurable dynamics and which normalizations
are possible. These issues are discussed in \cite{Gou18, TW12}.} 
(see \cite{KV19, La93, Le00, Vo99} for related related results).
In contrast Problem \ref{PrCLTAnyManifold} asks to construct 
a system where the CLT holds for most smooth functions.

\begin{problem}
Let $M$ be a compact manifold of dimension at least three. Does there exist a diffeomorphism of $M$ preserving a smooth measure which is $K$ but not Bernoulli?
\end{problem}

We note that in case of dimension two, the answer is negative due to Pesin theory \cite{BP07}. At present
there are no example of $K$ but not Bernoulli maps in dimension three. We refer the reader to
\cite{KRHV} for more discussion on this problem.

The next problem is motivated by Theorem \ref{ThZE-CLT}.

\begin{problem}
For which $\alpha$ does there exist a smooth system satisfying the CLT with normalization
which is regularly varying of index  $\alpha?$
\end{problem}

We mention that several authors 
\cite{Beck, BU18, CILB19, DS18}  obtained the Central Limit Theorem for circle rotations  
where normalization is a slowly varying function.
However, firstly, the functions considered in those papers are only piecewise smooth and, secondly,
there either requires an additional randomness or remove zero density subset of times.
Similar results in the context of substitutions are obtained in \cite{BBH14, PS19}.

\part{Central Limit Theorem}
\label{PtCLT}

\section{A criterion for CLT}
\label{ScProduct}

In order to prove our results, we 
use the strategy of \cite{Bolt} replacing Feller Lindenberg CLT for iid random variables by
a CLT for exponentially mixing systems due to \cite{BG}.
More precisely we need the following result.

\begin{proposition}
\label{PrBG}
 Let $\fm_T$ be a signed measure on $\reals^d$ 
and let $\cS_T:=\int_{\reals^d} A_t(G_t y) d\fm_T(t).$
Suppose that for $\|A_t\|_{C^1(Y)}$ is uniformly bounded, $\nu(A_t)\equiv 0$ and
\smallskip

(a) $\DS \lim_{T\to\infty} \|\fm_T\|=\infty$  where 
 $\|\fm\|$ is the total variation norm:\footnote{We remark that in all our applications
$\fm$ is a non-negative measure, so $\|\fm\|=\fm(\R^d).$}
$$ \|\fm\|=\max_{\reals^d=\Omega_1\cup \Omega_2}\{ \fm(\Omega_1)-\fm(\Omega_2)\};$$

(b) For each $r\in \naturals,$ $r\geq 3$ for each $K>0$ 
$$\lim_{T\to \infty} \int \fm_T^{r-1} (B(t,K\ln \|\fm_T\|)) d\fm_T(t)=0;$$

(c) There exists
$\sigma^2$ so that $\DS \lim_{T\to \infty} V_T=\sigma^2$, where
$$ V_T:=\int\cS_T^2(y)d\nu(y)= \iiint A_{t_1} (G_{t_1} y) A_{t_2} (G_{t_2} y) d\fm_T (t_1) \fm_T(t_2) d\nu(y).$$
Then $\cS_T$ converges  as $T\to\infty$ to normal 
distribution with zero mean and variance $\sigma^2.$
\end{proposition}

This proposition is proven in \cite{BG} in case $A_t$ does not depend on $t,$
however the proof does not use this assumption.

\section{The CLT for skew products}
\label{ScStandard}
\subsection{ Reduction to quenched CLT}
In this section we will prove Theorem \ref{ThCLT3}.
Consider first the case where 
\begin{equation}
\label{ZMeanFiber}
\int H(x, y)d\nu(y)=0
\end{equation}
 for each $x\in X.$ Given $x\in X$, we consider the measure 
\begin{equation}\label{eq:disc} 
\fm_N(x)=\frac{1}{\sqrt{N}} \sum_{n=0}^{N-1}  \delta_{\tau_n(x)}, \quad
A_{t,x}(y)=\frac{1}{\fm_N(x)(\{t\})} 
\sum_{n\leq N: \tau_n(x)=t} H(f^nx,y). 
\end{equation}

\begin{lemma}
\label{LmRandMes3}
Under the assumptions of Theorems \ref{ThCLT3},
there exists $\sigma^2$
(independent of $x$!)  and subsets $X_N\subset X$ such that
$\DS \lim_{N\to\infty} \mu(X_N)=1$ and for any sequence $x_N\in X_N$
the measures $\{\fm_N(x_N)\}$ satisfy the conditions of Proposition \ref{PrBG}.
\end{lemma}

The lemma will be proven later. Now we shall show how to obtain the CLT from the
lemma.

\begin{proof}[Proof of Theorem \ref{ThCLT3}]
Split 
\begin{equation}
\label{BaseFiber}
H(x,y)=\tH(x,y)+\brH(x)\quad\text{where}\quad \brH(x)=\int H(x,y) d\nu(y).
\end{equation}
 Note that
\begin{equation}
\label{eq:zerofiber}
 \int \tH(x,y) d\nu(y)=0.
\end{equation} Hence by
Lemma \ref{LmRandMes3},
$\DS \frac{1}{\sqrt{N}} \sum_{n=0}^{N-1} \tH(F^n(x,y))$ 
is asymptotically normal
and moreover its distribution is asymptotically independent of $x.$ On the other hand
by the CLT for $f$, 
$$ \frac{1}{\sqrt{N}} \sum_{n=0}^{N-1} \brH(\pi_x F^n(x,y))= 
\frac{1}{\sqrt{N}} \sum_{n=0}^{N-1} \brH(f^n(x)) $$
is also asymptotically normal
and its distribution depends only on $x$ but not on $y.$

It follows that 
$$  \frac{1}{\sqrt{N}} \sum_{n=0}^{N-1} \tH(F^n(x,y))\quad\text{and}\quad
\frac{1}{\sqrt{N}} \sum_{n=0}^{N-1} \brH(f^n(x)) $$ 
are asymptotically independent. 
Since the sum of two independent normal random variables is normal, the
result follows.
\end{proof}
\subsection{Proof of the quenched CLT (Lemma \ref{LmRandMes3})}
To prove Lemma \ref{LmRandMes3}, 
we need to check properties (a)--(c) of Proposition \ref{PrBG}. 
{\em Property (a)} is clear since $\|\fm_N(x)\|=\sqrt{N}.$ Other properties are less obvious and will be checked 
in separate sections below.

\subsubsection{{\bf Property (b)}}
Let 
\begin{equation}\label{eq:zz} X_{K, N}=\left\{x: \Card\{n:|n|<N \text{ and }\|\tau_n(x)\|\leq K\ln N\}\geq 
 N^{ 1/4-\epsilon}\right\}.
\end{equation}
\begin{lemma}\label{LmQuarter} Suppose that for some $\epsilon>0$ and
for each $ K$, $\DS \lim_{N\to\infty} N \mu(X_{K, N})=0$. Then 
there are sets $\hX_N$ such that for all $x_N \in \hX_N$ the measures
$\fm_N(x_N)$ satisfy property (b) and $\mu(\hX_N) \to 1$.
\end{lemma}

\begin{proof}
Given $K$ let
$\DS \hX_N(K)=\{x: f^n x\not\in X_{K, N}\text{ for } n<N\}. $
By the assumption of the lemma,  there exists $K_N\to \infty$ such that  $\mu(\hX_N)\to 1$, where $\hX_N:=\hX_N(K_N)$. 
Now we have for every $x\in \hX_N$:
$$ \int \fm_N^{r-1}(x)(B(t, K\ln N))d \fm_N(x)(t)=$$$$
\frac{1}{N^{r/2}} 
\sum_{n=0}^{N-1} \Card^{r-1} \{j<N:\; \|\tau_j (x)-\tau_n (x)\|\leq 
K_N \ln N\}\leq$$
$$ \frac{1}{N^{r/2}} 
\sum_{n=0}^{N-1} \Card^{r-1} \{j<N:\; \|\tau_{j-n}(f^nx)\|\leq K_N \ln N\}\leq  N^{(1/4-\epsilon)(r-1)-\frac{r}{2}+1} \to 0.$$
Here, in the last line
we used that $x\in \hX_N$ and  that $r \geq 3$.
\end{proof}

To finish the proof of property (b) it remains to show that if $\tau$ satisfies 
\eqref{eq:I1}, then for every fixed $K$, and for $\epsilon = 0.02$,
$\DS \lim_{N\to \infty}N\mu(X_{K, N})=0$.  
First observe that for large $N$
$$
X_{K,N} \subset
\{ 
x: \cL(x,N) \geq N^{0.22}
\},
$$
where
$$
\cL(x,N)  = \Card \{ n: N^{0.21} < |n| < N, \| \tau_n(x) \| \leq K \ln N \}.
$$
Next, observe that if $\tau$ satisfies \eqref{eq:I1}, then for every 
$n$ with $|n|\geq N^{0.21}$,
we have 
$$\mu( \| \tau_n \| <K\ln N)<Cn^{-5}.$$
We conclude by the Markov inequality that
$$
\mu(X_{K, N}) \leq N^{-0.22} \mu (\cL(x,N))
= N^{-0.22}  \sum_{n: N^{0.21} < |n| < N} \mu 
(\| \tau_n \| <K\ln N) < CN^{-1.06}.
$$
Property (b) follows.


\subsubsection{{\bf Property (c)}} 
We need to select $\sigma^2$ so that (c) holds.
Note that
$$ V_N(x)=\frac{1}{N}\int S_N^2(x, y) d\nu(y)=
\frac{1}{N}\sum_{n_1, n_2=1}^N \sigma_{n_1, n_2}(x) $$
where
$$ \sigma_{n_1, n_2}(x)=\int H(f^{n_1} x, G_{\tau_{n_1}(x)} y) 
H(f^{n_2}x, G_{\tau_{n_2}(x)} y) d\nu(y). $$
Thus
$$ \mu\left(V_N(x)\right)=\frac{1}{N} \sum_{n_1, n_2=1}^N \mu( \sigma_{n_1, n_2})
=\sum_{k=-N}^{N} \frac{N-|k|}{N} \int H(x, y) H(f^k x, G_{\tau_k(x)} y)  d\mu(x) d\nu(y). $$ 
Note that due to \eqref{ZMeanFiber} and exponential mixing of $G_t$,
\begin{equation}
\label{FiberEM}
 \left|\sigma_{0, k} (x) \right|\leq C \|H\|_{C^1}^2 e^{-c|\tau_k(x)|} .
\end{equation} 
If $\tau$ satisfies \eqref{eq:I1}, then the above implies that for some
\footnote{\label{FtBeta} The proof below requires only that $\beta>1$. The condition $\beta>2$ will
only be used in Section \ref{ScVar} to characterize the systems with zero asymptotic variance.}
 $\beta>2$
\begin{equation}
\label{Eq1Cor}
\int  \left|\sigma_{n, n+k} (x) \right| d\mu(x) = 
\cO\left(k^{-\beta}\right).
\end{equation}
In particular, \eqref{Eq1Cor} implies that the following limit exists
\begin{equation}
\label{DefSigma}
 \sigma^2:=\lim_{N\to\infty} \mu\left(V_N(x)\right)
=\sum_{k=-\infty}^{\infty}  \int \sigma_{0, k}(x)  d\mu(x).  
\end{equation}
To prove property (c) with $\sigma^2$ given by \eqref{DefSigma}, 
we note that for each $\eps$ there is $L$ such that
$$ V_N(x)=\frac{1}{N} \sum_{n=0}^{N-1} \sum_{k=-L}^L \sigma_{n, n+k}(x)+\cE_L(x) $$
where the error term satisfies $\|\cE_L\|_{L^1}\leq \eps.$
So it is enough to prove that for each fixed~$L$
$$ \lim_{N\to\infty} \frac{1}{N} \sum_{n=0}^{N-1} \sum_{k=-L}^L \sigma_{n, n+k}(x)=
\sum_{k=-L}^{L} \int \sigma_{0, k}(x) d\mu(x). $$
Since $\sigma_{n, n+k}(x)=\sigma_{0, k}(f^n x)$, the result follows from the ergodic theorem.

\section{Horocycle base}
\label{ScHoro}

\subsection{Reduction to a mixing local
limit theorem}
\begin{proof}[Proof of Theorem \ref{ThHoro}]

As in Section \ref{ScStandard} it suffices to give a proof under the assumption
\eqref{ZMeanFiber}. Indeed we can split arbitrary $H$ as 
$\DS H(x,y)=\brH(x)+\tH(x,y)$ where $\tH$ satisfies \eqref{ZMeanFiber}
and 
use the fact that due to \cite{FF03} $\brH_T(x)=O(T^\alpha)$ for some $\alpha<1.$

Analogously to \eqref{eq:disc}, we define
$$
\fm_T(x)=\frac{(\ln T)^{1/4}}{T} \int_{0}^T  \delta_{\tau_t(x)} dt, \quad
A_{t,x}(y)=\frac{1}{\fm_T(x)(\{t\})} 
\int_{s\leq T: \tau_s(x)=t} H(h_s x,y)ds.
$$
As before we check properties (a)--(c) of Proposition \ref{PrBG}. Property (a) is immediate
as $\|\fm_T\|=(\ln T)^{1/4}$. 

To prove (b) and (c) we need some preliminary information. Let us use the notation 
$x = (q,v) \in X$ and say that $q$ is the configurational component
of $x$.

Let $q_0\in Q$ and arbitrary reference point
and for each $q\in Q$ let $\Gamma_q$ be a shortest geodesic from $q_0$ to $q.$ Define
$\beta(q)=\int_{\Gamma_q} \lambda$ and let
$$ \xi_T(x)=\tau_T(x)-\beta(h_T x)+\beta(x) . $$
 \eqref{Wind} shows that $\xi_T(x)$ is an integral of $\lambda$ over a curve starting and ending at $q_0$,
so by \eqref{DualBase} it is an integer. 

We need the following extension of \cite[Theorem 5.1]{DS17}.
Let $\fg_T(x)$ be the configurational component of the
geodesic of length $\ln T$ starting at $q$ with speed $-v.$ Denote
$\DS s_T(x)=\left(\int_{\fg_T(x)}\lambda\right)+\beta(x)-\beta(\brx)$, where $ \brx=\bG_{-\ln T} x$
and $\bG_t$ denotes the geodesic flow.

We say that a function is piecewise continuous if the
set of discontinuities is contained in a finite union of proper compact submanifolds (with boundary).

\begin{proposition}
\label{PrTLTHoro}
There is a zero mean Gaussian density $\fp$, so that the following are true for all $x \in X$.

(a) For each $z \in \reals$, 
$$ 
\frac{1}{T}
\mes\left(t\leq T: \frac{\xi_t-s_T(x)}{\sqrt{\ln T}}\leq z\right)
=\int_{-\infty}^z \fp(s) ds+o(1).$$

(b)
For any set $A \subset X$ whose boundary is a finite union of proper compact submanifolds (with boundary), we have
\begin{equation}
\label{eq:MLLTcaseb}
\frac{\sqrt{\ln T}}{T} \int_0^T 1_{\xi_t(x) = k} 1_{h_t(x) \in A} dt
= \mu(A) \fp\left(\frac{k-s_T(x)}{\sqrt{\ln T}}\right)+o(1),
\end{equation}
where the convergence is uniform when $\frac{k-s_T(x)}{\sqrt{\ln T}}$ 
varies over a compact set.

(c) For any $k\in\integers$ and for any set A as in part (b),
$$  \mes(\{t\leq T: \xi_t(x)=k, x \in A\})
\leq \frac{C T}{\sqrt{\ln T}} \mu(A).$$
\end{proposition}
The proof of the proposition will be given in \S \ref{SSTLTLoc}.

Now we are ready to finish the proof of Theorem \ref{ThHoro}.
Property (b) of Proposition \ref{PrBG} now reduces to showing that for each $K$ and each $r\geq 3$
$$ \int \fm_T^{r-1} (B(t, K \ln \ln T)) d\fm_T(t) \to 0. $$
Observe that by Proposition \ref{PrTLTHoro}(a),  for each unit segment $I\subset \reals$, we have
$\DS \fm_T(I)\leq C/\ln^{1/4} T$ and hence $\fm_T(B(t, K \ln \ln T))\leq \frac{C(K) \ln \ln T}{\ln^{1/4} T}.$
Thus
$$  \int \fm_T(B(t, K \ln \ln T)) d\fm_T(t) \leq \frac{C^{r-1} (K) (\ln \ln T)^{r-1} }{\ln^{(r-1)/4} T} \|\fm_T\|_\infty 
\leq \frac{C^{r-1}(K) (\ln \ln T)^{r-1}}{\ln^{\frac{r-2}{4}} T} \to 0$$
since $r>2.$

To establish property (c) we need to compute $\DS \lim_{T\to\infty}  \frac{\sqrt{\ln T}\; \zeta(H_T^2)}{T^2}.$
We have
$$ \zeta(H_T^2)=\sum_{k_1, k_2\in \Z} \int \cI_{k_1, k_2}(x) d\mu(x) $$
where
$$ \cI_{k_1, k_2}(x)=\int_0^T \int_0^T 1_{\xi_{t_1}=k_1} 1_{\xi_{t_2}=k_2} \;
\rho(h_{t_1}x, h_{t_2} x, k_2-k_1+\beta(q_{t_2})-\beta(q_{t_1})) dt_1 dt_2,$$
$q_t$ is the configurational component of
$h_t(x)$ and
$$ \rho(x', x'', s)=\int H(x', y) H(x'', G_s y) d\nu(y). $$

Fix a large $R$ and partition the sum into three three parts. Let $I$ be the terms where
\begin{equation}
\label{HoroRegionI} |k_2-k_1|\leq R,\quad |k_1-s_T(x)|\leq R \sqrt{\ln T};
\end{equation}
$\RmII$ be the terms where $|k_2-k_1|>R$; and $\RmIII$ be the terms where
$$ |k_2-k_1|\leq R \quad \text{but}\quad |k_1-s_T(x)|>R \sqrt{\ln T}.$$
By our assumption, $\rho$ is exponentially small in $t$, uniformly in $x', x''$. Hence using 
the estimate 
$$ \mes\left(t_2\leq T: \xi_{t_2}(x)=k_2\right)\leq \frac{C T}{\sqrt{\ln T}} $$
valid by Proposition \ref{PrTLTHoro}(c) 
and summing over $k_2$ we obtain
$$ \left|\RmII\right|\leq 
\frac{C' T}{\sqrt{\ln T}} 
\sum_{k_1} \mes\left(t_1\leq T: \xi_{t_1}(x)=k_1\right) e^{-cR}\leq \frac{C'' T^2}{\sqrt{\ln T}} e^{-cR},
$$
$$ \left|\RmIII\right|\leq 
\frac{C' RT}{\sqrt{\ln T}} 
\sum_{|k_1-s_T(x)|>R \sqrt{\ln T}} 
\mes\left(t_1 \leq T: \xi_{t_1}(x)=k_1\right) $$
$$=
\frac{C' RT}{\sqrt{\ln T}} 
\mes\left(t_1 \leq T: |\xi_{t_1}(x)-s_T(x)|>R\sqrt{\ln T} \right). 
$$
Hence given $\delta$ we can take $R$ so large that both $\RmII$ and $\RmIII$ are smaller than
$\DS \frac{\delta T^2}{\sqrt{\ln T}}$ (for $\RmIII$ we use Proposition \ref{PrTLTHoro}(a)). 

Thus the main contribution comes from $I.$ To analyze the main term choose a small $\eps$ and 
divide $X$ into sets as in part (b) with diameter $< \eps$.
Let $x_l=(q_l, v_l)$ be the center of $\cC_l.$ Next we write 
$\cI_{k_1, k_2}=\sum_{l_1, l_2} \cI_{k_1, k_2, l_1, l_2}$ where
$$\cI_{k_1, k_2, l_1, l_2}=\int_0^T \int_0^T 
1_{\xi_{t_1}(x)=k_1} 1_{C_{l_1}} (h_{t_1} x)
1_{\xi_{t_2}(x)=k_2} 1_{C_{l_2}} (h_{t_2} x)
\rho(k_2-k_1+\beta(q_{t_2})-\beta(q_{t_1})) dt_1 dt_2. $$
Using uniform continuity of $\rho$, we obtain
\begin{equation}
\label{Ikkll}
\cI_{k_1, k_2, l_1, l_2}= \delta_{k_1, k_2, l_1, l_2}+\mes(\{t_1: \xi_{t_1}(x)=k_1, \cC_{l_1}\ni h_{t_1} x\})\cdot
\end{equation}
$$ 
\cdot\mes(\{t_2: \xi_{t_2}(x)=k_2, \cC_{l_2}\ni h_{t_2} x\}) 
\cdot\rho(x_{l_1}, x_{l_2}, k_2-k_1+\beta(q_{l_2})-\beta(q_{l_1}) )
$$
where the error term $\delta_{k_1, k_2, l_1, l_2}$ is smaller than
$\DS \frac{T^2 }{R^2 \ln T}\; \eps^6$ (here, the factor $\eps^6$ appears because by
Proposition \ref{PrTLTHoro}(b)
$$ \mes\left(t_j: \xi_{t_j}(x)=k_j, \; h_{t_j}(x)\in \cC_{l_j}\right)\leq C \frac{T}{\sqrt{\ln T}} \mu(C_{l_j})
\leq \brC \frac{T}{\sqrt{\ln T}} \eps^3 ). $$
Applying Proposition \ref{PrTLTHoro}(c) to the main term in \eqref{Ikkll} we get
that 
$$ \frac{\ln T}{T^2}\;  \cI_{k_1, k_2, l_1, l_2}\approx 
\mu(\cC_{l_1}) \mu(\cC_{l_2}) \rho(k_2-k_1+\beta(q_{l_2})-\beta(q_{l_1}))
\fp\left(\frac{k_1-s_T(x)}{\sqrt{\ln T}}\right)\fp\left(\frac{k_2-s_T(x)}{\sqrt{\ln T}}\right).
$$
Performing the sum of $l_1$ and $l_2$ we obtain
$\DS \frac{\ln T}{T^2}\;  \cI_{k_1, k_2}=$
$$\iint \rho(x', x'', k_2 - k_1 + \beta(q'')-\beta(q')) d\mu(x') d\mu(x'') 
\fp\left(\frac{k_1-s_T(x)}{\sqrt{\ln T}}\right)\fp\left(\frac{k_2-s_T(x)}{\sqrt{\ln T}}\right)+o_{\delta\to 0} (1) $$
where $x'=(q', v')$, $x''=(q'', v'').$
Performing the sum over $k_1, k_2$ as in \eqref{HoroRegionI} we obtain
$$ \frac{\sqrt{\ln T}\; \zeta(H_T^2)}{T^2}=
\left(\int_{-R}^R \fp^2 (z) dz\right)
 \sum_{|k|\leq R} \iint \rho(x', x'', \beta(q'')-\beta(q')+k) d\mu(x') d\mu(x'')+o_{R\to \infty}(1).$$
Letting $R\to\infty$ and using that for Gaussian densities
$\DS  \int_{-\infty}^{\infty} \fp^2 (z) dz=\frac{\fp(0)}{\sqrt{2}}$
we get
\begin{equation}
\label{HoroVar}
 \lim_{T\to\infty}  \frac{\sqrt{\ln T}\; \zeta(H_T^2)}{T^2}=\sigma^2:=
\frac{\fp(0)}{\sqrt{2}} \sum_{k\in \integers} \iint \rho(x', x'', \beta(q'')-\beta(q')+k) d\mu(x') d\mu(x'').
\end{equation}
This completes the proof of property (c) and establishes Theorem \ref{ThHoro}.
\end{proof}

\subsection{Mixing local limit theorem for geodesic flow}
\label{SSTLTLoc}
\begin{proof}
[Proof of Proposition \ref{PrTLTHoro}]
Part (a) is \cite[Theorem 5.1]{DS17} but we review the proof as it will be needed
for parts (b) and (c). The key idea is to rewrite the temporal limit 
theorem for the horocycle flow as a central limit theorem for the geodesic flow.
To be more precise, let $\fh (x,t)$ and $\fg (x,t)$ denote the configurational component
of the horocycle $\cH(x,t)$ and the 
geodesic of length $t$ starting from $x$. Consider the quadrilateral $\Pi(x,t, T)$ formed by 
$$ \fh(x, t), \; -\fg(h_t(x), T), \; -\fh(\bG_{-\ln T}  x, t/T),\;  \fg(x, T) $$
where $-$ indicates that the curve is run in the opposite direction. 
This curve $\Pi(x, t, T)$ is contractible as can be seen by shrinking $t$ and $T$ to zero.
Therefore the Stokes Theorem gives
\begin{equation}
\label{eq:Stokes}
\xi_t(x)=\left(\int_0^{\ln T} \tau^*(\bG_r h_u \brx) dr\right)+\beta(h_u \brx)-\beta(\brx) 
\end{equation}
where $\brx=\bG_{-\ln T} x,$ $u=t/T$ and $\tau^*(q,v)=\lambda(v). $
Note that if $t$ is uniformly distributed on $[0, T]$ then $u=t/T$ is uniformly distributed 
on $[0,1]$.
Since the curvature is constant, it follows that 
$h_u \brx$ is uniformly distributed on $\cH (\brx, 1)$.
Now part (a) follows from the central limit theorem for the geodesic flow $\bG$.

To prove part (b), write 
$$
\hat \tau_S(y) = \int_0^S \tau^*(\bG_r y) dr + \beta(y) - \beta(\bG_S y).
$$
Then by \eqref{eq:Stokes}, we have
$$
\frac{\sqrt{\ln T}}{T} \int_0^T 1_{\xi_t(x) = k} 1_{h_t(x) \in A} dt = 
\sqrt{\ln T} \int_0^1 1_{\hat \tau_{\ln T} (h_u(\bar x)) = k} 
1_{\bG_{\ln T}(h_u(\bar x)) \in A} du
$$
\begin{equation}
\label{eq:horgeo}
=  \sqrt{\ln T} \int_0^1 1_{\hat \tau_{\ln T} (\tilde x) = k} 
1_{\bG_{\ln T}(\tilde x) \in A} d m_{\cH(\brx,1)} (\tilde x).
\end{equation}
where $m_{\cH}$ is the arc-length parametrization of $\cH$.
Let us represent the geodesic flow $\bG$ as a suspension over a 
Poincar\'e section $M$
such that $\cT : M \to M$, the first return map to $M$ is Markov (\cite{B73})
and let $\tau_0$ be the first return time. Now we can apply 
\cite[Theorem 3.1(B)]{DN-Flows} to conclude that \eqref{eq:horgeo} 
is asymptotic the RHS of \eqref{eq:MLLTcaseb}. Although
that theorem is formulated for measures absolutely continuous w.r.t $\mu$
but the proof is the same for the measure $m_{\cH(\brx,1)}$ as well.
Note that all assumptions
of that theorem are immediate except for the following: there is no proper subgroup
of $\integers \times \reals$ that would support a function in the cohomology class
of $(\int_0^{\tau_0} \tau^*(\bG_r (.))du,\tau_0(.))$ (with respect to the map $\cT$).
However, this statement follows from
\cite[Lemma A.3]{DN16}. Thus we have established part (b).

Note that the approach of \cite{DN-Flows} also allows to lift the anticoncentration
inequality from the map $\cT$ to the flow $\bG$. Since $\cT$ is a subshift 
of finite type, the anticoncantration inequality holds (see \cite[Lemma A.4]{DDKN}).
Thus
we obtain the anticonentration inequality for $\bG$, which is part (c)
of the proposition.
\end{proof}

\section{Variance} 
\label{ScVar}
In order to complete the proofs of Theorems \ref{ThZE-CLT} and \ref{ThWeakMix} we need to
show that the variances for the examples from \S \ref{SSEx} are not identically zero.
This will be done in \S \ref{SSVNZ} while in \S \ref{SSZVHom} we will discuss a characterization 
of vanishing variance for some of our systems.

\subsection{Observables with non-zero asymptotic variance}
\label{SSVNZ}
Here we show that for the systems in Theorem  \ref{ThHoro}
and Corollary \ref{CrUE-CLT},
there exist observables with non-zero asymptotic variance.

 One simple observation is that if the base system satisfies the classical CLT,  then we can take an observable 
which depends only on $X$ and, by Definition \ref{DefCLT}, the asymptotic
 variance $\sigma^2(\cdot)$ is typically non zero.

In the setting of Corollary \ref{CrUE-CLT}, \eqref{DefSigma} shows that,
for functions satisfying \eqref{ZMeanFiber}
 the asymptotic variance is
given by
\begin{equation}
\label{SkewGK}
\sigma^2=\sum_{k=-\infty}^\infty \int\int H(x, y) H(f^k x, G_{\tau_k(x)} y) d\nu(y) d\mu(x). 
\end{equation}
By ergodicity of $f$, for each $p$ the set of $p$ periodic points has measure 0. Thus for
each $p$ and for almost every $x_0$, there is some $\delta >0$ such that $f^j B(x_0, \delta)\cap B(x_0,\delta)=\emptyset$
for $0<|j|\leq p.$ Let us fix $x_0$ so that
$\varkappa(x_0)$ is positive and finite, 
where $\varkappa$ is the density of $\mu$ with respect to the volume.
Let $\phi$ be a non negative function supported on the unit interval. 
Set $\DS H(x,y)=\phi\left(\frac{d(x, x_0)}{\delta}\right) D(y)$ 
where $D$ is a smooth observable on $Y.$
Then the term in \eqref{SkewGK} corresponding to $k=0$ equals to
$$ \delta^{a} \varkappa(x_0) \int_{\reals^a} \phi^2 (d(\bx, 0)) d\bx \; [\nu(D)]^2 (1+o_{\delta\to 0} (1))$$
where $a=\dim(X)$.
The terms with $0<|k|\leq p$ are equal to zero since for such $k$, the function 
$\DS \phi\left(\frac{d(x, x_0)}{\delta}\right) \phi\left(\frac{d(f^k x, x_0)}{\delta}\right)$ 
is identically equal to $0.$ For $|k|>p$, we can integrate with
respect to $y$ and get that the $k$-th term in \eqref{SkewGK} is 
$O(\delta^a \theta^{|k|})$ 
with some $\theta < 1$ by the exponential mixing
of $G$.
Summing over $k$ we see that the non-zero $k$'s contribute
$O(\delta^a \theta^p)$. Therefore for 
$p$ sufficiently large and $\delta$ sufficiently small,
the RHS of \eqref{SkewGK} is positive.

A similar argument shows that the variance defined in \eqref{HoroVar} is not identically zero.
Again we fix a small $\delta$ and let
$H(x,y)=\phi\left(\frac{d(q, q_0)}{\delta}\right) D(y) $ where $D$ is as above.
Then for small $\delta$ if $q', q''$ are in the support of $\phi\left(\frac{d(\cdot, q_0)}{\delta}\right)$
then 
$$\rho(x', x'', k+\beta(q')-\beta(q''))\approx \phi\left(\frac{d(q', q_0)}{\delta}\right)
\phi\left(\frac{d(q, q_0)}{\delta}\right) \int D(y) D(G_k y) d\nu(y). $$
It follows that 
$$ \sigma^2\approx 
\delta^4
\frac{\fp(0)}{\sqrt{2}} \left(\int_{\reals^2}   \phi (d(\bx, 0)) d\bx\right)^2
 \bsigma^2(D) \text{ where } \bsigma^2(D)=\sum_{k=-\infty}^\infty \int D(y) D(G_k y) d\nu(y). 
$$
It remains to observe that $\bsigma^2(D)$ is non-zero for typical $D,$ (as follows, for 
example, from the discussion in \S \ref{SSZVHom}).

\subsection{Zero Variance and homology}
\label{SSZVHom}
Here we present more information about functions with vanishing asymptotic variance. 
We  recall two useful results. We formulate the results for discrete time systems,
but similar results hold for flows.

\begin{proposition}[{{\em $L_2$--Gotshalk-Hedlund Theorem}}]
\label{PrGHL2}
Let $\cF$ be an automorphism of a space $\cM$ preserving a 
measure $\fm.$ Let $\cA:\cM\to\reals$ be a zero mean observable such that
$\DS \left\|\sum_{n=0}^{N-1} \cA\circ \cF^n\right\|_{L^2}$ is bounded. Then
there exists an $L^2$ observable $\cB$ such that
\begin{equation}
\label{CoB}
\cA=\cB\circ \cF-\cB.
\end{equation}
\end{proposition}

The next result helps to verify the conditions of the above proposition.
Let $\brho_n=\int \cA(\bx) \cA(\cF^n \bx) d\fm(x).$

\begin{proposition}
\label{PrGK-Cob}
Suppose that 
\begin{equation}
\label{C1SD}
\sum_{n=0}^\infty n |\brho_n|<\infty.
\end{equation}
Then
$\left\|\sum_{n=0}^{N-1} \cA\circ \cF^n\right\|_{L^2}$ is bounded iff 
$$\Sigma^2(\cA):=\sum_{n=-\infty}^\infty \brho_n=0. $$
\end{proposition}

\begin{proof}
The result follows because
$$ \left\|\sum_{n=0}^{N-1} \cA\circ \cF^n\right\|_{L^2}^2=\sum_{n=-N}^N (N-|n|) \brho_n=
N\Sigma^2-\sum_{n=-N}^N n \brho_n-\sum_{|n|\geq N} N \brho_n $$
and both sums in the last expression are less than
$\DS \sum_{n=-\infty}^\infty |n\brho_n |.$
\end{proof}

Now we describe  application of Propositions \ref{PrGHL2} and \ref{PrGK-Cob}.

(a) {\em Systems from Corollary \ref{CrUE-CLT}}.
Notice that \eqref{Eq1Cor} (and $\beta>2$) implies that \eqref{C1SD} holds for observables $H(x,y)$ satisfying
\eqref{ZMeanFiber}. Splitting a general $H$ as in 
\eqref{BaseFiber} and applying Propositions \ref{PrGHL2} and \ref{PrGK-Cob} to $\tH,$
we conclude that the asymptotic variance vanishes iff $\tH$ is an $L^2$ coboundary,
that is, iff $H$ is a {\em relative coboundary} in the sense that $H$ can be decomposed as
$\DS  H(x,y)=I\circ F-I+\brH$
where $I, \brH$ are in $L^2$ and $\brH$ does not  depend on $y.$


(b) {\em Systems with Anosov base.} Assume that a base is an Anosov diffeo. Then 
Theorems 4.1 and 4.7 in \cite{DDKN} tell us that \eqref{C1SD} holds if either the mean of
$\tau$ is non-zero or if $d\geq 5$, so the asymptotic variance vanishes iff $H$ is an $L^2$
coboundary. If we suppose that $\|\tau\|_{C^1}$ is small, and that the drift $\mu(\tau)$ is not
on the boundary of the Weyl chamber, that is $\chi(\mu(\tau))\neq 0$ for any root in the Lie algebra,
then the system will be partially hyperbolic and, for generic $\tau$, it will be accessible
(cf. \cite{Br75, BW99}). Then the results of \cite{Wil13} will imply that $H$ is a continuous coboundary.
 Therefore the integral of $H$ with respect to any $F$ invariant measure is zero, implying that the set of coboundaries is a subspace of infinite codimension.

\part{Higher rank Kalikow systems}
\label{PtKalikow}

\section{Homogeneous abelian actions} \label{sec:cartan}
Let $H$ be a connected nilpotent or semi-simple Lie group, $\Gamma$ be a co-compact lattice
and $Y=H/\Gamma$. Let $\Ab=\mathbb{Z}^d$ or $\mathbb{R}^d$.   If $H$ is nilpotent, i.e. $Y$ is a nilmanifold, then we consider the action $G$  on $Y$ by affine maps. If $H$ is semisimple, 
and $\Gamma$ is a co-compact irreducible lattice, then $G$ acts on $Y$ by left translations.



Let $(G,\Ab, Y,\nu)$ be an abelian action on $Y$ as above.
 Let $d_H$ denote the right-invariant metric on $H$ and $d_Y$ the induced metric on $Y$.
For ${\bf t} \in \Ab$, the corresponding diffeomorphism $G_\bt$ will be denoted by ${\bf t}$ for simplicity. Moreover, ${\bf t}_\ast:TY\to TY$ denotes the differential of ${\bf t}$.

By classical Lyapunov theory, there are finitely many linear functionals $\chi_i:\Ab \to \R$ and a splitting 
$\DS TY:=\bigoplus_{i=1}^mE^{\chi_i}$  which is invariant under $\Ab$, such that for any $\epsilon>0$, 
there exists a Riemannian metric $\|\cdot\|_{TY}$, such that for all ${\bf t}\in \Ab$, we have
  \begin{equation}\label{eq:exp}
 e^{\chi_i({\bf t})-\epsilon\|{\bf t}\|}\|v\|_{TY}\le \|{\bf t}_{\ast}(v)\|_{TY}\le e^{\chi_i({\bf t})+\epsilon\|{\bf t}\|}\|v\|_{TY}, \text{ for every } v\in E_{\chi_i}.
 \end{equation} If $\Ab=\Z^d$, we extend the functionals $\{\chi_i\}$ to $\R^d$. It follows that there exist transverse, $G$-invariant foliations $W^{i}=W_{\chi_i}$ such that for every $y\in M$, 
 $W^i(y)\subset Y$ is a smooth immersed submanifold and  
\begin{equation}\label{eq:foliA}
 TY:=\bigoplus_{i}TW^i.
 \end{equation}
The map $y\mapsto W^i(y)$ is smooth. In the algebraic case that we are considering the spaces $E_{\chi_i}$ and $W^i$ are also algebraic: let $\mathfrak h$ be the Lie algebra of $H$, then the tangent space $TM$ at $e\Gamma$ is identified with $\mathfrak h$, and there exist subalgebras $\mathfrak h_i$ of $\mathfrak h$ such that $\DS \mathfrak h=\bigoplus_i\mathfrak h_i$, and $\mathfrak h_i=E_{\chi_i}$ under the identification. Accordingly, there exist subgroups $H_i=\exp(\mathfrak h_i)$ such that $W^i(y)=H_i(y)$ for any $y\in Y$. If $G$ is an $\R^k$ action on homogeneous spaces of noncompact type, the derivative action on $TY$ induced by $G$ is identified with the adjoint action. 
 The connected components of $\mathbb{R}^d$ where all Lyapunov functions keep the
same sign are called {\em Weyl chambers}. The Lie identity implies that in each Weyl chamber $C$
the subspaces
$$ \fh_C^+=\sum_{\lambda_i>0 \text{ on } C} \fh_i, \quad
\fh_C^-=\sum_{\lambda_i\leq 0 \text{ on } C} \fh_i $$
are subalgebras, hence integrable. We denote the corresponding foliations by
$\bbW_C^+$ and $\bbW_C^-$ respectively.

If there exists a nonzero Lyapunov functional, then we call $G$ a (partially) hyperbolic action, and if the foliation $W^c$ corresponding to zero Lyapunov functionals coincides with the orbit foliation, then we call 
$G$ {\it Anosov action}. In particular, for actions $G$ as in {\bf a1} 
the center foliation $W^c$ is trivial. 

For partially hyperbolic actions as in {\bf a2}, the assumption that $G$ is identity on the center space means that  the center foliation is generated by an action of the group $H^c$ which commutes with  
$G$:
\begin{equation}\label{cen:id}
\text{If }\bry\in W^c(y), \;\bry=g_c\cdot y, \; g_c\in H^c \;\; \text{ then }\;\; G_t(\bry)=g_c\cdot G_t(y).
\end{equation}

 Both Cartan actions (Example \ref{eg1}) and Weyl Chamber flows (Example \ref{eg2}) are Anosov actions.

We introduce a system of local coordinates on $Y$ using  the exponential map from 
$\DS TY=\bigoplus_{i=1}^mE^{\chi_i}$ to $Y$. Thus we can rewrite the vector 
$z\in T_y Y$ as $(z_1,\cdots, z_m)$, where $z_i\in E^{\chi_i}$. There exists a constant $\zeta_0$ such that for any $y\in Y$, the exponential map $\exp: B({\bf 0},\zeta_0)\subset T_y Y\to Y$ is one to one.
For $\delta_i\leq \zeta_0$, $i\leq m$, let 
\begin{equation}\label{eq:para}C(\{\delta_i\},y):=\{\exp(z):z=(z_1,\cdots,z_m)\in B({\bf 0},\zeta_0)\subset T_yM, \ |z_i|\leq \delta_i/2\}
\end{equation}
denote the parallelogram centered at $y$ with side lengths $\{\delta_i\}$.


Recall that  a smooth action $G$ on $(Y,\nu)$ is {\em exponentially mixing for 
sufficiently smooth functions} if there exists $k\in \N$ such that for all $f,g\in C^k(Y)$,
$\bv\in \Ab$
$$
|\langle f,g\circ G_{\bf v}\rangle- \nu(f)\nu(g)|\leq Ce^{-\eta \|{\bf v}\|}\|f\|_k\|g\|_k.
$$
By \cite{KKS,KS1}, any action $G$ as in {\bf a1} or 
{\bf a2} is exponentially mixing for sufficiently smooth functions.

Moreover, we say that $g$ is  {\em exponentially mixing on balls} if\ there exist $C,\eta',\eta>0$ such that for every ${\bf v}\in \Ab$, every $B(y,r), B(y',r')\subset Y$ with $y,y'\in Y$ and $r,r'\in (e^{-\eta' \|{\bf v}\|},1)$ the following holds:

\begin{equation}
\label{ExpMixBalls}
|\nu(B(y,r)\cap G_{\bv}B(y',r))- \nu(B(y,r))\nu(B(y',r'))|\leq Ce^{-\eta \|{\bf v}\|}.
\end{equation}

A standard approximation argument (see eg. \cite{Ga07}) shows  that exponential mixing for sufficiently smooth functions implies that $G$ is exponentially mixing on balls. So we have: 
\begin{lemma}\label{exp:mix} Any action $G$ as in {\bf a1.} or {\bf a2.} 
is exponentially mixing on balls.
\end{lemma}

\section{Relative atoms of the past partition}
Recall that $F:(\Sigma_A\times Y,\mu\times \nu)\to(\Sigma_A\times Y,\mu\times \nu)$ is given by $F(\omega,y)=(\sigma \omega, G_{\tau(\omega)}y).$ Let $\mathcal P_{\epsilon}$ be a partition of $\Sigma_A$ given by cylinders on coordinates $[-\epsilon^{-\frac{1}{\beta}},0]$, where $\beta$ is the H\"older exponent of $\phi$. Let $\mathcal{Q}_\epsilon$ be a partition of $Y$ into sets with piecewise smooth boundaries and of diameter $\leq \epsilon$.

Let $\Omega$ denote the alphabet of the shift space 
$\Sigma_A = \Omega^{\mathbb Z}$. 
For $\omega^- = (..., \omega_{-1}, \omega_0) 
\in \Omega^{\mathbb Z_{\leq 0}}$, let 
$$\Sigma_A^+(\omega^-) = 
\{ \omega^+ = (\omega_1,\omega_2,...) \in \Omega^{\mathbb Z_+}: (...,\omega_{-1}, \omega_0, \omega_1,...)
\in \Sigma_A\}.
$$
Note that $\Sigma_A^+(\omega^-)$ only depends on finitely many coordinates of 
$\omega^-$. 
We will also use the notation $\omega = (\omega^-, \omega^+)$
and $\Sigma_A^+(\omega) = \Sigma_A^+(\omega^-)$.
For $\omega = (\omega^-, \omega^+)$
and $S^+ \subset \Sigma_A^+(\omega) $, we write
$$\mu_{\omega}^+(S^+) =\mu( \{(\omega^-, \bar \omega^+): \bar \omega^+ \in S^+\}).$$
With a slight
abuse of notation, we also denote by $\mu_{\omega}^+$ a measure on
$\Sigma_A$ defined by 
$\mu_{\omega}^+(S) = \mu_{\omega}^+(\{\bar \omega^+:
 ( \omega^-, \bar \omega^+) \in S\})$.
Notice that we have, for any measurable $S\subset \Sigma_A$,$$\mu(S)=\int_{\Sigma_A}\mu^+_\omega(S)d\mu(\omega).$$

We can assume that $\tau$ only depends on
the past. Indeed, if this is not the case, then $\tau$ is cohomologous
to another H\"older function $\bar \tau$ depending only on the past:
$\tau(\omega) = 
\bar \tau(\omega^-) + h(\omega) - h(\sigma \omega)$. If 
$\bar F$ is the $T,T^{-1}$ transformation constructed using 
$\bar \tau$ and $H(\omega, y) = (\omega, G_{h(\omega)}y)$,
then $H\circ F = \bar F \circ H$. Since $F$ and $\bar F$ are conjugate,
we can indeed assume that $\tau$ only depends on the past.

The main result of this section is:
\begin{proposition}\label{prop:atoms}
There exists $\epsilon_0>0$ and a full measure set $V\subset \Sigma_A\times Y$ such that for every 
$(\omega,y)\in V$, the atoms of 
$$
\bigvee_{i=0}^{\infty}F^{i}(\mathcal P_{\epsilon_0}\times \mathcal{Q}_{\epsilon_0})
$$
are of the form $\{\omega^-\times
\Sigma_A^+(\omega^-)\}\times \{y\}$, i.e. the past of $\omega$ and the $Y$-coordinate are fixed.
\end{proposition}

Before we prove the above proposition, we need some lemmas. For a non-zero $\chi_i$, 
let $\mathcal{C}_i\subset \R^d$ be a cone 
$$\cC_i=\{\a\in \R^d: \chi_i(\a)\geq c'\|\a\|\} , \quad \text{where}\quad c'=\min_{i: \chi_i \neq 0}\|\chi_i\|/2.$$ 

We start with the following lemma:
\begin{lemma}\label{lem:ball} Let $(G,Y,\nu)$ be as in a1. or a2. 
Choose cones $\hat\cC_i$ properly contained in $\cC_i.$
Let $\{\a_j\}_{j\in \N}\subset \Ab$, 
be a sequence such that $\a_1=0$ and
\begin{itemize}
\item[A.] $\sup_{j}\|\a_{j+1}-\a_j\|<+\infty$; 
\item[B.] 
for every $i$ we have
$\DS \sup_{j: {\bf a}_j\in { \hat\cC_i}} \|{\bf a}_j\|=\infty$.
\end{itemize}
 Then there exists $\eta=\eta(G,\sup_{j}\|\a_{j+1}-\a_j\|)>0$ 
 such that for any $y,y'\in Y$ with $y'\notin W^c(y)$, there exists $j\in \N$ such that 
 $d_Y(\a_jy,\a_jy')\geq \eta/4$.
\end{lemma} 
In order to prove the above lemma, we need the following:
\begin{lemma}\label{lem:uncov} 
Let $H^c<H$ be the subgroup of $H$ such that $W^c(y)=H^c y$ for any $y\in Y$. Then  
$\exists\breta>0$ such that for any $y,y'\in H$ with $y'\notin H^c(y)$ and any $\{\a_j\}$ satisfying $A.,B.$, there exists $j_0$ such that 
$$
d_H(\a_{j_0}y,\a_{j_0}y')> \breta. $$
\end{lemma}
\begin{proof}
 Fix $y,y'\in H$. WLOG, assume $d_H(y,y')<\zeta_0$. We can write $y=\exp(Z)y'$, where $Z\in\mathfrak h$, and $Z=\bigoplus_i Z_i$  with $Z_i\in\mathfrak h_i$. Since $y'\notin H^c(y)$, there exists $i$ such that $\chi_i \neq 0$ and $Z_i\neq 0$. Accordingly there is a Weyl chamber $\cC$ such that splitting
$Z=Z^++Z^-$ with $Z^\pm\in \fh_\cC^\pm$ we have $Z^+\neq 0.$
Let $y''=\bbW_\cC^-(y)\cap \bbW_\cC^+(y').$ Then $y''\neq y'$ since $Z\not\in \fh_\cC^-.$

Let $\hat\cC$ be a cone which is strictly contained inside $\cC.$ Note that by the definition of $y''$, there exists a global constant $K>0$ such that  for each $\alpha_j\in \cC$ we have
$d_H(\alpha_j y, \alpha_j y'')\leq K \zeta_0$. By triangle inequality, $d_H(\alpha_jy,\alpha_jy')\geq d_H(\alpha_jy',\alpha_jy'')-d_H(\alpha_jy,\alpha_jy'')$. It is enough to notice that  due to the fact that the vectors in $\fh_\cC^+$ 
are expanded by $\hat\cC$ at a uniform rate and $\DS \sup_{j: \a_j\in \hat\cC} \|\a_j\|=\infty$,
there exists $j$ such that $d_H(\a_j y', \a_j y'')\geq K\zeta_0 +\breta,$ for some $\breta>0$.
\end{proof}

With Lemma \ref{lem:uncov}, we can prove Lemma \ref{lem:ball}:
\begin{proof}[Proof of Lemma \ref{lem:ball}]
Since $\Gamma\subset H$ is co-compact, it follows that there exists $c>0$ such that 
\begin{equation}\label{eq:zzs1}
\inf_{y\in H}\inf_{\gamma\neq e}d_H(y,y\gamma)>c>0.
\end{equation}
Let $\DS C_1:=\sup_{j}\|\a_{j+1}-\a_j\|<\infty$ and let $C=C(\alpha)>0$ be such that 
\begin{equation}\label{eq:zna11}
\sup_{0<d_H(y,y')\leq 1}\sup_{\|{\bf b}\|<C_1}\frac{d_H({\bf b}y,{\bf b}y')}{d_M(y,y')}\leq C,
\end{equation}
 Let $0<\eta<\breta$ be such that $c\geq (C+1/4)\eta$ (recall that $\breta$ is the constant from
Lemma \ref{lem:uncov} ).
Let $y,y'\in Y$, with $y'\notin W^c(y)$, with $d_N(y,y')\leq \eta/4$. 
By taking appropriate lifts of $y$ and $y'$ to $H$, we can assume that $d_H(y,y')\leq \eta/4$. Notice that by Lemma \ref{lem:uncov}, there exists $j_0\in \N$ such that $d_H(\a_{j_0}y,\a_{j_0}y')>\eta/4$. Let us take the smallest $j_0$ with this property. Then, $d_H(\a_{j_0-1}y,\a_{j_0-1}y')\leq \eta/4$. Therefore by the bound in \eqref{eq:zna11}
$$
d_H(\a_{j_0}y,\a_{j_0}y')=d_H\Big((\a_{j_0}-\a_{j_0-1})(\a_{j_0-1}y),(\a_{j_0}-\a_{j_0-1})(\a_{j_0-1}y')\Big)\leq C\eta.
$$
 Take $\gamma\in H$ such that $d_M(\a_{j_0}y,\a_{j_0}y')=d_H\Big(\a_{j_0}y,\a_{j_0} y'\gamma\Big)$.
By \eqref{eq:zzs1} we get
$$
d_H\Big(\a_{j_0}y,\a_{j_0} y'\gamma\Big)
\geq d_H\Big(\a_{j_0}y',\a_{j_0} y'\gamma\Big)-d_H\Big(\a_{j_0}y,\a_{j_0} y'\Big)
\geq c-C\eta\geq \eta/4.
$$
This finishes the proof.
\end{proof}

 Recall that for $\tau:\Sigma_A\to \Z^d$ (or $\R^d$) and $n\in \N$,  we denote 
 $\DS \tau_n(\omega):=\sum_{j=0}^{n-1}\tau(\sigma^j\omega),$
and $\DS \tau_{-n}(\omega)=-\tau_{n}(\sigma^{-n}\omega)$.
 The next result, proven in \S \ref{SSCones},
 helps to verify the conditions of Lemma \ref{lem:ball}(B). 
\begin{lemma}\label{lem:RW} Let $\tau:\Sigma_A\to G$ be a 
H\"older function that is not cohomologous to a function
taking values in a linear subspace of $G$ of dimension $<d$. Then for any cone $\cC\subset \R^d$, 
for $\mu$ a.e. $\omega\in \Sigma_A$
$$
\sup_{v \in \{\tau_n(\omega)\}_{n\in \mathbb Z _{\pm}}\cap \cC} \|v\| = \infty.
$$
\end{lemma}

With all the above results, we can now prove Proposition \ref{prop:atoms}. 

\begin{proof}[Proof of Proposition \ref{prop:atoms}]
We will take $\epsilon_0$ smaller than $\eta/4$. Notice that if $G$ is as in {\bf a1}, then $W^c$ is trivial and therefore, for every 
$y \neq y'\in Y$, $y'\notin W^c(y)$. Let $\a_j:=\tau_{-j}(\omega)$. 
Then by Lemma \ref{lem:RW}, there exists a full measure set of $\omega$ such that $\{\a_j\}$ belongs to every $\mathcal{C}_i$ infinitely often
and the norm of such $\a_j$'s is unbounded. 
Moreover, $\DS \sup_{j}\|\a_{j+1}-\a_j\|<\sup |\tau|$. Therefore the assumptions of Lemma \ref{lem:ball} are satisfied. In particular it follows that if $y\neq y'$, then there exists $j$ such that 
$G_{\a_j}(y)$ and $G_{\a_j}(y')$ are not in the same atom. 
This finishes the proof in case {\bf a1}.

 Let now $G$ satisfy {\bf a2}. Let $\mathfrak{k}\subset \mathfrak{h}$ be the maximal compact subalgebra.
Take a small $\delta.$
By further decreasing $\epsilon_0$ we can assume that the following holds: 
for every $\cW\in \mathfrak{k}\setminus\{0\}$ with $\|\cW\|\leq \delta$,
 there exists an atom $Q\in \mathcal{Q}_{\epsilon_0}$ satisfying
\begin{equation}\label{eq:qinv}
Q\text{ is not invariant under the automorphism }g_\cW=\exp(\cW). 
\end{equation}
Let us first prove that such $\epsilon_0$ exists. If not, then for every $\epsilon>0$ there exists $\cW_\epsilon\in \mathfrak{k}$, such that $\|\cW_\epsilon\|_{TN}\leq \delta$ and
every atom of $\mathcal{Q}_\epsilon$ is invariant under $g_{\cW_{\epsilon}}$. 
 Then for each $n\in\naturals,$ every atom of $\mathcal{Q}_\epsilon$ is also invariant under 
 $g_{n \cW_{\epsilon}}$. Taking $n_\eps=[\delta/\|\cW_\eps \|]+1$, $\tilde\cW=n_\eps \cW_\eps$ 
 we get that
$\|\tilde\cW_\eps\|\in [\delta, 2\delta]$ such that every atom of $\mathcal{Q}_\epsilon$ is invariant 
under $g_{\tilde\cW_{\epsilon}}$.

By compactness (since atoms of $\mathcal{Q}_\epsilon$ shrink to points) and taking $\epsilon\to 0$, it would follow that there exists $\cW_0\in \mathfrak{k}$ with 
$\|\cW_0\|_{TY}\in [\delta, 2\delta]$ such that $g_{\cW_0}=id$. 
 If $\delta>0$ is sufficiently small, this gives a contradiction and finishes the proof of \eqref{eq:qinv}.

By Corollary 2 in \cite{Gui89}, the skew product is ergodic. Let $\Lambda$ be the subset of points whose forward (and also backward) orbit is dense. Hence, $\zeta(\Lambda)=1$.

Notice that if $(\omega,y)\in \Lambda$, and $(\omega,y), (\bar\omega,y')$  lie in the same atom of 
$\DS \bigvee_{i=0}^{\infty} F^i(\mathcal P \times \mathcal{Q}_{\epsilon_0})$, then 
$\omega^-=\bar\omega^-$. Since $\tau$ depends only on the past, $\tau_{-j}(\omega)=\tau_{-j}(\bar\omega)$ for $j\in \N$. We will show that $y'=y$.

 Assume first that $y'\in W^c(y)$ and let $y'=g_c\cdot y$, $g_c=\exp(\cW)$, with $\cW\neq 0$. 
 If $\cW\in \mathfrak{k}$, let $Q=Q_\cW$ be such that \eqref{eq:qinv} is satisfied and if $\cW\notin \mathfrak{k}$, let $Q$ be any atom $\mathcal{Q}$. Note that there exists $q\in Q$ and $\epsilon=\epsilon(g_c)>0$ such that  $B(\epsilon,q)\subset Q$ and $g_c\cdot B(\epsilon,q)\cap Q=\emptyset$. 
 Indeed, if not then $Q$ would be invariant under the translation by $g_c=\exp(\cW)$. 
 If $\cW\in \mathfrak{k}$ we get a contradiction with \eqref{eq:qinv}. If $\cW\notin \mathfrak{k}$ then  the set $\{g_c^n\;:\; n\in \Z\}$ is not compact in 
$H$ and by Moore ergodicity theorem \cite{Mo}, the automorphism $g_c$ is ergodic, a contradiction.  This contradiction shows that such $q$ and $\epsilon$ exist.  
 
 Since the $F$ orbit of $(\omega,y)$ is dense, there
exists $n$, such that $F^{-n}(\omega,y)\in \Sigma_A\times B(\epsilon,q)\subset \Sigma_A\times Q$. Let $u=\phi_{-n}(\omega).$ Then by \eqref{cen:id}, $G_{u}y'=g^c G_{u}y\notin Q$. So  $F^{-n}(\omega,y)$  and $F^{-n}(\omega',y')$ are not in the same atom of $\mathcal{P}\times \mathcal{Q}$.

If $y'\notin W^c(y)$ then 
we again use Lemma  \ref{lem:RW} to finish the proof.
\end{proof}

\begin{remark}
We believe that ALL partially hyperbolic algebraic abelian actions satisfy the assertion of Proposition \ref{prop:atoms}. However, the proof is more complicated if there is a polynomial growth in the center. We plan to deal with the general situation in a forthcoming paper.
\end{remark}

\section{Non Bernoulicity under zero drift. Proof of Theorem \ref{thm:main}}
\subsection{The main reduction}
We introduce the notion of $(\epsilon,n)$-closeness which is an averaged version of Bowen closeness. Let $d$ denote the product metric on $\Sigma\times Y.$
 Two points $(\omega,y),( \omega',y')\in \Sigma_A\times Y$ are called 
$(\epsilon,n)$-{\em close} if 
$$\#\left\{i\in[1,n]:d\big(F^i(\omega,y),F^i(\omega',y')\big)<\epsilon\right\}\ge (1-\epsilon)n.$$
We will now state two propositions that imply Theorem \ref{thm:main}.

\begin{proposition}\label{prop:VWB}
If $F$ is Bernoulli then for every $\epsilon,\delta>0$ there exists $n_0$ such that for every $n\geq n_0$ there exists a measurable set $W\subset \Sigma_A\times Y$ with $ \zeta(W)>1-\delta$ 
such that if $(\omega,y),(\bar \omega,\bar y)\in W$, there exists a map 
$\Phi_{ (\omega^-, y) (\bar \omega^-, \bar y)}:\Sigma ^+_A(\omega)\to \Sigma ^+_A(\bar \omega)$ with $(\Phi_{\omega^-,\bar \omega^-})_*(\mu^+_{\omega})=\mu^+_{\bar \omega}$ and a set $U_{\omega^-}\subset \Sigma^+_A(\omega)$ such that:
\begin{itemize}
\item[(1)] $\mu^+_{\omega}(U_{\omega^-})>1-\delta$;
\item[(2)] if $z\in U_{\omega^-}$ then $((\omega^-, z), y)$ and 
$((\bar\omega^-, \Phi_{ (\omega^-, y) (\bar \omega^-, \bar y)} z), \bar y)$
are $(\epsilon, n)$-close.
\end{itemize}
\end{proposition}

We will also need another result. For $\epsilon>0$, $n\in \N$, $\omega\in \Sigma_A, y'\in Y$, let 
\begin{equation}\label{eq:de}
D(\omega,y',\epsilon,n):=\Big\{y\in Y\;:\; \exists\, \omega'\in \Sigma_A\;\text{s.t.}\; (\omega,y) \text{ and } (\omega',y')\text{ are } (\epsilon,n)\text{-close}\Big\}.
\end{equation}

\begin{proposition}\label{prop:cruc} There exists $\epsilon'>0$, an increasing sequence $\{n_k\}$, a family of sets $\{\Omega_k\}$, $\Omega_k\subset \Sigma_A$, $\mu(\Omega_k)\to 1$, such that 
$$
\lim_{k\to\infty}\sup_{\substack{\omega\in \Omega_k\\ y'\in Y}}\nu(D(\omega,y',\epsilon',n_k))=0.
$$
\end{proposition}

We will prove Proposition \ref{prop:VWB} in a \S \ref{SSHB}
and Proposition \ref{prop:cruc} in \S \ref{SS-Cr1-Cr2}. Now we show how these two propositions imply Theorem \ref{thm:main}:
\begin{proof}[Proof of Theorem \ref{thm:main}] We argue by contradiction. 
Fix $\epsilon=\epsilon' / 100$, $\delta=\epsilon$, and let $n=n_k$ (for some sufficiently large $k$, specified below). Let $W\subset \Sigma_A\times Y$ be the set from Proposition \ref{prop:VWB}.
Let 
$$W^y:=\{\omega\in \Sigma_A\;:\; (\omega,y)\in W\}\quad\text{and}\quad W_\omega:=\{y\in M\;:\; (\omega,y)\in W\}. $$ 
By Fubini's theorem, there exists $Z\subset \Sigma_A$, $\mu(Z)\geq 1-2\epsilon$ such that for every $\omega\in Z$, 
${ \nu}(W_\omega)>1/2$. Let $k$ be large enough (in terms of $\epsilon$) such that $\mu(Z\cap \Omega_k)\geq 1-4\epsilon$. By Fubini's theorem, it follows that there exists $Z'\subset Z\cap \Omega_k$, $\mu(Z')>1-4\epsilon$ such that for $\omega\in Z'$, $\mu_{\omega}^+(Z\cap \Omega_k)>1-8\epsilon$. 
In particular, it follows that 
$$
\mu_{\omega}^+(\{\bar\omega^+\in U_{\omega^-}:(\omega^-,\bar\omega^+)\in Z\cap\Omega_k\})>1-16\epsilon.$$
 Let $\omega=(\omega^-,\omega^+)\in Z\cap \Omega_k\cap( \{\omega^-\}\times U_{\omega^-})$ and
let $(\bar \omega,y')\in W$. Since $\omega\in Z$ it follows that ${ \nu}(W_\omega)>1/2$.  Since $\omega\in \Omega_k$, it follows that
for $k$ large enough  there exists 
\begin{equation}\label{ynotin}y\in W_\omega\setminus D(\omega,y',\epsilon',n_k). 
\end{equation}
 Since $\omega^+\in U_{\omega^-}$, by $(2)$ we get that $(\omega^-,\omega^+,y)$ and $(\bar\omega^-,\Phi_{\omega^-,\bar\omega^-}(\omega^+),y')$ are $(\epsilon, n_k)$-close. This by the definition of $D(\omega,y',\epsilon',n_k)$ implies that $y\in D(\omega,y',\epsilon',n_k)$. This however contradicts \eqref{ynotin}.
This contradiction finishes the proof.
\end{proof}

\subsection{Hamming--Bowen closeness}
\label{SSHB}
We start with introducing the notion of VWB (very weak Bernoulli) partitions in the setting of skew-product for which the assertion of Proposition \ref{prop:atoms} holds (see eg. \cite{SB} or \cite{KRHV}):
Let $\mathcal{R}$ be a partition of $\Sigma_A\times Y$. Two points $(\omega,y),( \omega',y')\in \Sigma_A\times Y$ are called $(\epsilon,n,\mathcal{R})$-{\em matchable} if 
$$\#\{i\in[1,n]:F^i(\omega,y)\text{ and }F^i(\omega',y')\text{ are in the same } \mathcal{R} \text{ atom}\}\ge (1-\epsilon)n.$$
\begin{definition}\label{def:VWB} $F$ is very weak Bernoulli with respect to $\mathcal R$ if and only if for every $\epsilon'>0$, there exists $n'$ such that for every $n\geq n'$ there exists a measurable set $W'\subset \Sigma_A\times M$ with $\mu\times \nu(W')>1-\epsilon'$ such that if $(\omega,y),(\bar \omega,\bar y)\in W'$, there exists a map 
$\Phi_{ (\omega^-, y) (\bar \omega^-, \bar y)}:\Sigma ^+_A(\omega)\to \Sigma ^+_A(\bar \omega)$ with $(\Phi_{\omega^-,\bar \omega^-})_*(\mu^+_{\omega})=\mu^+_{\bar \omega}$ and a set $U'_{\omega^-}\subset \Sigma^+_A(\omega)$ such that:
\begin{itemize}
\item[(1)] $\mu^+_{\omega}(U'_{\omega^-})>1-\epsilon'$;
\item[(2)] if $z\in U'_{\omega^-}$ then $((\omega^-, z), y)$ and 
$((\bar\omega^-, \Phi_{ (\omega^-, y) (\bar \omega^-, \bar y)} z), \bar y)$
are $(\epsilon', n,\mathcal{R})$-matchable.
\end{itemize}
\end{definition}


\begin{proof}[Proof of Proposition \ref{prop:VWB}]
 Recall that by \cite{Orn1} if $F$ is Bernoulli then  it is VWB 
 with respect to every non-trivial partition.



Let 
$(\mathcal P\times  Q)_n$ be the sequence of partitions defined above, where the atoms have diameter that goes to $0$ as  $n\to \infty$.  Let $\bar{n}$ be such that the atoms of $(\mathcal P\times  Q)_{\bar{n}}$ have diameter $\leq \epsilon$. This then implies that if two points $(\omega,y)$ and $(\omega',y')$ are $(\epsilon, n)$ matchable, then they are $(\epsilon,n)$-close. It is then enough to use VWB definition for $(\mathcal P\times  Q)_{\bar n}$ with $\epsilon'=\min\{\delta,\epsilon\}$. This finishes the proof.
\end{proof}

\begin{remark} Now we explain why it is easier to work with closeness rather than matchability, 
in the case $G=\R^d$. Notice that if $(\omega,y)$ and $(\omega',y')$ are $(\epsilon,n)$-close, and $\|u\|<\delta<\epsilon$, then $(\omega,y)$ and $(\omega',G_u y')$ are $(\epsilon+\delta,n)$ close.
\footnote{Notice that for any $i\in \N$ the points $F^i(\omega',y')$ and $F^i(\omega',G_u y')$ are $\delta$ close. Indeed, they have the same first coordinate and the second one is $G_{\tau_i(\omega)}y'$ 
vs $G_{u+\tau_i(\omega)}y'$ which are $\delta$ close since $\| u \|<\delta$. 
}
 This is not necessarily true for matchability (if the orbit of $y'$ is always close to the boundary of the partition). This property of closeness crucially simplifies our consideration as it allows us
 to obtain a crucial inclusion \eqref{eq:pert}.
\end{remark}

\subsection{Proof of Proposition \ref{prop:cruc}}
\label{SS-Cr1-Cr2}

Given $\Omega_k, n_k$ denote
$$a_{k}(\epsilon'):=\sup_{\substack{\omega\in \Omega_k\\ y'\in Y}}\nu(D(\omega,y',\epsilon',n_k)).
$$
\begin{proposition}\label{prop:cruc2} 
There exists $n_1\in \N$ and  a family of sets $\{\Omega_k\}$ (as above) such that if 
$\DS \epsilon_k:=\left(1-\frac{1}{50k^2}\right)\epsilon_{k-1}$, $\epsilon_1:=\frac{1}{10n_1}$ and $n_{k+1}=(10k)^{100}\cdot n_k$, then we have 
$$a_{k}(\epsilon_k)\to 0,\;\text{as }k\to \infty.$$
\end{proposition}

We remark that the recursive relations in Proposition \ref{prop:cruc2} imply that
\begin{equation}
\label{EpsK}
 \epsilon_k=\epsilon_1 \prod_{j=2}^k \left(1-\frac{1}{50j^2}\right),
\end{equation}
\begin{equation}
\label{N-K}
n_{k+1}=n_1 (10^kk!)^{100}.  
\end{equation}

 Proposition \ref{prop:cruc2} which is proven in Section \ref{ScProofCruc2} 
immediately implies Proposition \ref{prop:cruc}:

\begin{proof}[Proof of Proposition \ref{prop:cruc}]
We define 
$\DS \epsilon':=\inf_{k\geq 1}\epsilon_k=\prod_{j=2}^\infty \epsilon_1 \left(1-\frac{1}{50j^2}\right). $
Then by the definition of $\{\epsilon_k\}$, $\epsilon'>0$ and monotonicity, we have 
$$
0\leq a_k(\epsilon')\leq a_k(\epsilon_k)\to 0,
$$
as $k\to \infty$. This finishes the proof.
\end{proof}

\section{Consequence of exponential mixing}
We have the following quantitative estimates on independence of the sets $D(\omega,y',\epsilon', n_k)$ under the $\Ab$ action $G$, $\Ab\in\{\Z^d,\R^d\}$. This is the only place in the proof where we use 
the exponential mixing of $G.$

\begin{lemma}\label{lem:induc} For $k\in \N$ let $\omega_1,\omega_2\in \Sigma_A$ be such that 
\begin{equation}\label{eq:BS}
\sup_{r\leq n_{k-1}}\|\tau_r(\omega_i)\|\leq 2k^{20}n_{k-1}^{1/2}
\end{equation}
for $i=1,2$. Then, for any $y_1,y_2\in Y$, any $v\in \Z^d$, $\|v\|\geq k^{25}n_{k-1}^{1/2}$,
and any $\epsilon>0$.
$$
\nu\Big(G_v(D(\omega_1,y_1,\epsilon,n_{k-1}))\cap D(\omega_2,y_2,\epsilon,n_{k-1})\Big)\leq 
C_\#
\cdot \prod_{i=1,2}\nu\Big(D(\omega_i,y_i,\epsilon+2^{-n_{k-1}^{1/2}},n_{k-1})\Big).
$$
\end{lemma}
\begin{proof} 
 Let $\DS L:=\max\{\sup_{\|v\|=1}\| G_v\|_{C^1},100\}$. 
Then if $d(y,y')\leq 
{(2L)}^{-2k^{20}n_{k-1}^{1/2}}$, then 
$$d(G_v y, G_v y')\leq {L}^{2k^{20}n_{k-1}^{1/2}}\cdot {(2L)}^{-2k^{20}n_{k-1}^{1/2}} 
\leq 2^{-2k^{20}n_{k-1}^{1/2}}\leq 2^{-n_{k-1}^{1/2}}$$
for all $v\in \Ab$ with $\|v\|\leq 2k^{20}n_{k-1}^{1/2}$. 
Using this for $v=\|\tau_r(\omega_i)\|$, $r<n_{k-1}$, \eqref{eq:BS} 
implies that if $d(y,y')\leq 
{(2L)}^{-2k^{20}n_{k-1}^{1/2}}$, then 
\begin{equation}\label{eq:BS2}
d(G_{\tau_j(\omega_i)}(y), G_{\tau_j(\omega_i)}(y'))\leq 2^{-n_{k-1}^{1/2}}, \text{ for all } j<n_{k-1}.
\end{equation}

Therefore for every $y\in D(\omega_i,y_i,\epsilon, n_{k-1})$, 
\begin{equation}\label{eq:ins}
B\left(y,{(2L)}^{-2k^{20}n_{k-1}^{1/2}}\right)\subset D(\omega_i,y_i,\epsilon+2^{-n_{k-1}^{1/2}},n_{k-1}).
\end{equation}

Using Besicovitch theorem for the cover $\left\{B\left(y,{(2L)}^{-2k^{20}n_{k-1}^{1/2}}\right)\right\}$, where 
$$y\in D(\omega_i,y_i,\epsilon,n_{k-1}),$$
 we get a finite cover by a family of balls $\{B^{j,i}_s\}_{j\leq C', s\leq m_j}$ $i=1,2$, such that for 
 every $i\in \{1,2\},$ $j\leq C'$, the balls $\{B^{j,i}_s\}_{s\leq m_j}$ are pairwise disjoint. Therefore 
$$
\nu\Big(G_v(D(\omega_1,y_1,\epsilon,n_{k-1}))\cap 
D(\omega_2,y_2,\epsilon,n_{k-1})\Big)\leq  
\sum_{j,j'}\sum_{s,s'}\nu(G_v(B^{j,1}_s)\cap B^{j',2}_{s'}).
$$

Notice that $\|v\|\geq k^{25}n_{k-1}^{1/2}$ and so $e^{-\eta'' v}\leq  (\frac{1}{2L})^{2k^{20}n_{k-1}^{1/2}}$. Using that $G$ is exponentially mixing on balls  in the sense of
\eqref{ExpMixBalls},  we get that the above term is upper bounded by

\begin{equation}\label{eq:z1a1}
C\cdot \sum_{j,j'}\sum_{s,s'}\nu(B^{j,1}_s)\nu(B^{j',2}_{s'})=C\left[\sum_j\sum_s \nu(B^{j,1}_s)\right]\cdot \left[\sum_{j'}\sum_{s'} \nu(B^{j',2}_{s'})\right].
\end{equation}
Since the balls are disjoint for fixed $i$ and $j$, we have   
$$
\sum_s\nu(B^{j,i}_s)=\nu\left(\bigcup_{s}B^{j,i}_s\right)\leq \nu(D(\omega_i,y_i,\epsilon+2^{-n_{k-1}^{1/2}},n_{k-1}))
$$
where the last inequality follows from \eqref{eq:ins}. Since the cardinality of $j'$s is globally bounded (only depending on the manifold $N$), \eqref{eq:z1a1} is upper bounded by
$$
C\cdot C_d\cdot \prod_i \nu(D(\omega_i,y_i,\epsilon+2^{-n_{k-1}^{1/2}},n_{k-1})).
$$
This finishes the proof.
\end{proof}

We also have the following lemma.

\begin{lemma}\label{lem:aste} 
For any constant $C_2>1$ the following is true.
If $n_1>C_2$ and $b_k$ is a sequence of
real numbers satisfying
$$
b_1\leq \Big(\frac{1}{100n_1}\Big)^{300d}  \text{ and }
b_k\leq C_2\cdot n_k^{2d+1}\cdot b_{k-1}^2,
$$
then $b_k \to 0$.

\end{lemma}
\begin{proof} 
By induction, we see that
$$
\ln b_k \leq (2^{k-1} -1) \ln C_2 + (2d+1) \left[ 
\sum_{l=2}^{k} 2^{k-l} \ln n_l
\right] + 2^{k-1} \ln b_1
$$
Now using \eqref{N-K}, we obtain
$$
\ln b_k \leq (2^{k-1} -1) \ln C_2
+ (2d+1) \left[ 
\sum_{l=2}^{k} 2^{k-l} 100 l (\ln 10 + \ln l)
\right] + 2^{k+2}d \ln n_1 + 2^{k-1} \ln b_1.
$$
Using the condition on $b_1$, the result follows.
\end{proof}

\section{Construction of $\Omega_k$}
\label{eq:om}

Let $n_1$ be a number specified below and $n_k$ be defined by \eqref{N-K}.
For $k\geq 2$ define
$$
A_k:=\Big\{\omega\in \Sigma_A\;:\; \#\{(i,j)
 \in 
[0,(10k)^{100}]\times [0,(10k)^{100}],i\neq j \;:\; $$
$$  \frac{1}{(|j-i|n_{k-1})^{1/2}}\|\tau_{(j-i)n_{k-1}}(\sigma^{in_{k-1}}\omega)\|\geq k^{-20}\}>(10k)^{200}(1-k^{-9})\Big\}, 
$$
$$
B_k:=\Big\{\omega\in \Sigma_A:\; \#\{i< (10k)^{100}:\;
\sup_{r\leq n_{k-1}} \frac{1}{n_{k-1}^{1/2}}\|\tau_{r}(\sigma^{in_{k-1}}\omega)\|\leq k^{20}\}>(10k)^{100}(1-k^{-9})\Big\}.
$$

For $\omega\in \Sigma_A$, let $\omega_{[0,n-1)}$ denote the cylinder 
in coordinates $[0,\ldots,n-1)$ determined by $\omega$ and let
$$
\tilde{A}_k=\bigcup_{\omega\in A_k}\omega_{[0,n_k-1)}\text{ and } \tilde{B}_k=\bigcup_{\omega\in B_k}\omega_{[0,n_k-1)}.
$$
This way, $\tilde{A}_k$ and  $\tilde{B}_k$ are unions of cylinders of length $n_k$. 

The next lemma is proven in \S \ref{ScSep}.

\begin{lemma}\label{lem:cl} 
For any $C_0>0$, there exists an $n_0$, such that if $n_1\ge n_0$, we have:
\begin{itemize}
\item[{\bf m1}.] for every $k \geq 1$,\; 
$\min\left(\mu(\tilde{A}_k), \mu(\tilde{B}_k)\right)\geq 1-C_0k^{-8}.$

\item[{\bf m2}.] for every $\omega \in \tilde{A}_k$, 
\begin{multline}\label{eq:n1}
\#\Big\{(i,j)
 \in 
[0,(10k)^{100}]\times [0,(10k)^{100}],i\neq j \;:\;\\ \frac{1}{(|j-i|n_{k-1})^{1/2}}\|\tau_{(j-i)n_{k-1}}(\sigma^{in_{k-1}}\omega)\|\geq k^{-20}/2\Big\}>(10k)^{200}(1-k^{-9})
\end{multline}
and for every $\omega\in \tilde{B}_k$,
\begin{equation}\label{eq:n2}
\#\left\{i< (10k)^{100}\;:\;  \sup_{r\leq n_{k-1}} \frac{1}{n_{k-1}^{1/2}}\|\tau_r(\sigma^{in_{k-1}}\omega)\|\leq 2k^{20}\right\}>(10k)^{100}(1-k^{-9}).
\end{equation}
\end{itemize}
\end{lemma}

Define
\begin{equation}\label{eq:om1}
\Omega_1:=\left\{\omega\;:\; \|\tau_{n_1}(\omega)\|\geq n_1^{1/2-1/10}\right\}.
 \end{equation}
We suppose that $n_1$ is large enough, see below. 
For $k\geq 2$ we define:
$$ \Omega_k:=\tilde{A}_k\cap\tilde{B}_k\cap \Big\{\omega\in \Sigma_A\;:\;\#\{i< (10k)^{100}\;:\;   \sigma^{in_{k-1}}(\omega)\in \Omega_{k-1}\}>(10k)^{100}(1-k^{-5})\Big\}. $$

\begin{lemma}\label{lem:cyl}
 For every $k$, the set $\Omega_{k}$ is a union of cylinders of length $n_k$.
 \end{lemma}
 \begin{proof}For $k=1$, this follows from the definition of $\Omega_1$ as $\tau$ 
 only depends on the past.  Also by definition the sets $\tilde{A}_k$ and $\tilde{B}_k$ are  unions of cylinders of length $n_k$. Now inductively, if $\Omega_{k-1}$ is a union of cylinders of length $n_{k-1}$, then for every $i< (10k)^{100}$, the event   $\sigma^{in_{k-1}}(\omega)\in \Omega_{k-1}$, depends only on the $[in_{k-1},(i+1)n_{k-1}]$ coordinates of $\omega$. Since $i<(10k)^{100}$, the union of these events depends only on the first $n_k$ coordinates of $\omega$.
\end{proof}

Let $\bC_k = \{\mathcal C: \mathcal C$ is a union of cylinders of length $n_{k-1}\}.$ Now since $\mu$ is Gibbs, 
there exists a constant $C_1 \ge 1$ independent of the cylinders $\mathcal C$ and 
 of $k$
such that for any cylinders $\mathcal C_1,\mathcal C_2\in \bC_k$, for any $m\geq n_{k-1}$
$$ \mu(\mathcal C_1\cap \sigma^m \mathcal C_2) \leq C_1\mu(\mathcal C_1)\mu (\mathcal C_2).$$
 We obtain by induction that for any $\mathcal C_1,\dots,\mathcal  C_\ell \in \bC_k$, any $j_1<\dots <j_\ell$,
\begin{equation}
\label{QInd}
 \mu\left(\bigcap_{i=1}^\ell \sigma^{j_i n_{k-1}} \mathcal C_i\right) \leq  C_1^\ell
 \prod_{i=1}^\ell \mu(\mathcal C_i).
 \end{equation}
We assume that $n_1$ is so large that $\mu(\Omega_1)\geq 1-C_1^{-2} 2^{-200}$.

\begin{proposition}\label{prop:mes} 
There exists a constant $C_0>0$, such that for any $k\ge 1$,
\begin{equation}\label{eq-es-o}
\mu(\Omega_k)\geq 1-C_0k^{-7}.
\end{equation}
\end{proposition}

\begin{proof}[Proof of Proposition \ref{prop:mes}:]
Set $\DS C_0=\frac{1}{C_1^2 \;20^{200}}.$ 
We prove \eqref{eq-es-o} by induction. By the choice of $n_1$ and $C_0$, \eqref{eq-es-o} holds for $k=1$. Now assume it holds for $k-1\ge 1$. We are going to show it holds for $k$.

We claim that 
$\mu(D_k)\leq C_0k^{-7}/3$, where
$$
D_k=\Big\{\omega\in \Sigma_A\;:\;\#\{i< (10k)^{100}\;:\;   \sigma^{in_{k-1}}(\omega)\in \Omega_{k-1}\}<
(10k)^{100} - 
(10k)^{95}
 \Big\}.
$$
By Lemma \ref{lem:cyl}, the set $\Omega_{k-1}$ is a union of cylinders of length $n_{k-1}$. So is the complement $\Omega_{k-1}^c$.

Divide the interval $[0, (10k)^{100}]$ into $10(10k)^{94}$ intervals of length $10^5k^6.$
If $\omega\in D_k$, one of those intervals $I$ 
should contain at least $k$ visits to $\Omega_{k-1}^c.$
Let $i_1, \dots i_{k}$ be the times of the first $k$ visits inside $I.$  
By \eqref{QInd}, for each tuple $i_1,\dots, i_{k}$ 
$$ \mu\left(\sigma^{i_j n_{k-1}} \omega \in \Omega_{k-1}^c\text{ for } j=1,\dots ,k\right)
\leq (C_1\mu(\Omega_{k-1}^c))^{k}. $$
Since the number of tuples inside $I$ is less than $|I|^{k}=10^{5k}k^{6k}$,
$$ \mu\left(\#\{i\in I: \sigma^i \omega\in \Omega_{k-1}^c\}\geq k\right)
\leq (10k)^{6k}C_1^{k}\mu(\Omega_{k-1}^c)^{k}. $$
Since there are $10(10k)^{94}$ intervals, we have
$$ \mu(D_k) \leq 10(10k)^{94}(10k)^{6k}C_1^k\mu(\Omega_{k-1}^c)^{k}\le \frac{1}{C_1^k2^{100k}k^k}\le  C_0k^{-7}/3.$$
By {\bf m1} in Lemma \ref{lem:cl} and by
the definition of $\Omega_k$, we obtain $\mu(\Omega_k)\ge 1-C_0k^{-7}$.
\end{proof}

\begin{definition}\label{def:toas} We say that a pair $(i,j)\in[0,(10k)^{100}]^2$ is $n_k$--good (for $\omega$) if 
for $v\in \{i,j\}$
$\sigma^{v n_{k-1}}\omega \in \Omega_{k-1}$,
\begin{equation}\label{eq:zz1}
\frac{1}{(|j-i|n_{k-1})^{1/2}}\|\tau_{(j-i)n_{k-1}}(\sigma^{in_{k-1}}\omega)\|\geq k^{-20}/2,
\end{equation}
and
\begin{equation}\label{eq:zz2}
\sup_{r\leq n_{k-1}} \frac{1}{n_{k-1}^{1/2}}\|\tau_r(\sigma^{v n_{k-1}}\omega)\|\leq 2k^{20}.
\end{equation}
 \end{definition}
By definition  of $\Omega_k$, there are at least $(10k)^{200}(1-5k^{-5})$ $n_k$--good 
 pairs  $(i,j)$, for every $\omega\in \Omega_k$.

\section{Proof of Proposition \ref{prop:cruc2}}
\label{ScProofCruc2}

We will show that Proposition \ref{prop:cruc2} holds for sets $\Omega_k$ and $n_1$ from Section \ref{eq:om}. Let $C_2=
10^{200}\cdot
C_\#\cdot d^d\cdot 100^d(\sup  \|\tau\|)^d$, where $C_\#$ is from Lemma \ref{lem:induc}.

 We start with the following lemma:
\begin{lemma}\label{lem:a1} Let $n_1>C_2$ be sufficiently large. Then 
$$
a_1(\epsilon_1)\leq \Big(\frac{1}{100n_1}\Big)^{300d}.
$$
\end{lemma}
\begin{proof} Let $\omega\in \Omega_1$ and let $y\in D(\omega,y',\epsilon_1,n_1)$. 
Thus there is some $\omega'$ so that
$(\omega,y)$ and $(\omega',y')$ are $(\epsilon_1,n_1)$-close. Since $\epsilon_1=\frac{1}{10 n_1}$ it follows that for every $0\leq i\leq n_1-1$, 
$$
d\Big(F^i(\omega,y),F^i(\omega',y')\Big)<\epsilon_1.
$$

Since $\tau$ depends only on the past and is H\"older continuous with exponent $\beta$,  this implies in particular that 
$$
 \|\tau_i(\omega)-\tau_i(\omega')\|\leq C\epsilon_1^{\beta}\text{ for } i\leq n_1.
$$
Let $\epsilon_0=\epsilon_1^{\beta}$. Using closeness on the second coordinate, we get
\begin{equation}\label{eq:zd1}
d\Big(G_{\tau_i(\omega)}y,\;  G_{\tau_i(\omega)}y'\Big)<2C\epsilon_0\text{ for } i\leq n_1.
\end{equation}

We claim that \eqref{eq:zd1} implies that 
\begin{equation}\label{eq:new1}
d_H\Big(G_{\tau_i(\omega)}y,\; G_{\tau_i(\omega)}y'\Big)<2C\epsilon_0\text{ for } i\leq n_1.
\end{equation}
Indeed, if not let $i_0\leq n_1$ be the smallest 
index $i$ for which  \eqref{eq:new1} doesn't hold. This means that 
$$
d_H\Big(G_{\tau_{i_0-1}(\omega)}y, \; G_{\tau_{i_0-1}(\omega)}y'\Big)<2C\epsilon_0.
$$
Note that by \eqref{eq:zd1} there is some $\gamma$ so that
$$
d_H\Big(G_{\tau_{i_0}(\omega)}y, \; G_{\tau_{i_0}(\omega)}y'\gamma\Big)<2C\epsilon_0,
$$
and by the definition of $i_0$, $\gamma \neq e$.
The last two displayed inequalities imply that for some global constant $C''>0$,
$$
d_H\Big(G_{\tau_{i_0}(\omega)}y',\; G_{\tau_{i_0}(\omega)}y'\gamma\Big)<C''\epsilon_0.
$$
 If $\epsilon_0$ is small enough, this gives a contradiction with the systole bound \eqref{eq:zzs1}. So \eqref{eq:new1} indeed holds. 
 
Since $\omega\in \Omega_1$ (see \eqref{eq:om1}),  it follows that  
\begin{equation}\label{eq:n_1}
\|\tau_{n_1}(\omega)\|\geq n_{1}^{1/2-1/10}.
\end{equation}

It follows that $G_{\tau_{n_1}(\omega)}$
expands the leaves of one of the Lyapunov foliations by at least
$\DS e^{c n_{1}^{2/5}}$. 
Hence each leaf intersects the set of $y'$ satisfying \eqref{eq:new1}
in a set of measure $\DS O\left(e^{-c n_{1}^{2/5}}\right). $

 
 Therefore $\nu(D(\omega,y',\epsilon_1,n_1))\leq C'\cdot 
 e^{ -c n_1^{2/5}}$, whence 
$\DS
a_1(\epsilon_1)\leq C\cdot e^{ -c n_1^{2/5}}\leq \Big(\frac{1}{100n_1}\Big)^{300d}
$
if $n_1$ is sufficiently large. The proof is finished.
\end{proof}

 The next result constitutes a key step in the proof.

\begin{lemma}\label{lem:tpr}
For any $k\in \N$, any $\omega\in \Omega_k$, any $y'\in M$ and any $y\in D(\omega,y',\epsilon_k,n_k)$, there exists $(i_{k-1},j_{k-1})\in [1,(10k)^{100}]^2$, such that $|i_{k-1}-j_{k-1}|\geq  (10k)^{95}$, $(i_{k-1},j_{k-1})$ is $n_k$ good  (see Definition \ref{def:toas})  and 
there are $u_k, v_k$ such that
$\|u_k\|\leq (\sup |\tau|)n_k,$ $\|v_k\|\leq (\sup |\tau|)n_k,$ and
$$G_{\tau_{i_{k-1} n_{k-1}}(\omega)} y \in D\left(\sigma^{i_{k-1} n_{k-1}} \omega, 
G_{u_k} y', \Big(1-\frac{1}{100k^4}\Big)\epsilon_{k-1}, n_{k-1}\right), $$
$$G_{\tau_{j_{k-1} n_{k-1}} (\omega)} y \in D\left(\sigma^{j_{k-1} n_{k-1}} \omega, 
G_{v_k} y', \Big(1-\frac{1}{100k^4}\Big)\epsilon_{k-1}, n_{k-1}\right).
$$
\end{lemma}

Before we prove the above lemma, let us show how it implies Proposition \ref{prop:cruc2}.

\begin{proof}[Proof of Proposition \ref{prop:cruc2}]
Let $\Lambda_k=\{u:\|u\|\le (\sup|\tau|) n_k, \, 100 dn_ku\in\mathbb Z^d\}$. 
It is easy to see that $\#\Lambda_k=(100d(\sup|\tau|)n_k^2)^d$. 
Notice that for any $\ell_k$ with $\|\ell_k\|\leq n_k$ there exists $\ell\in \Lambda_k$ such that $\|\ell_k-\ell\|\leq n_k^{-1}$. Therefore, for any $\bar{\omega}\in \Sigma_A$
\begin{equation}\label{eq:pert}
D\left(\bar{\omega}, G_{\ell_k} y', \Big(1-\frac{1}{100k^4}\Big)\epsilon_{k-1},n_{k-1}\right)\subset D\left(\bar{\omega}, G_{\ell}\,y', \delta_{k-1},
n_{k-1}\right)
\end{equation}
where
$\delta_{k-1}:= \Big(1-\frac{1}{100k^4}\Big)\epsilon_{k-1}+\frac{1}{n_k}$.
Now combining Lemma \ref{lem:tpr} and \eqref{eq:pert} with the choice
$\ell_k\in \{u_k,v_k\}$ where $u_k,v_k$ are from Lemma \ref{lem:tpr}, we deduce
\begin{equation}\label{eq:tpr2}
D(\omega,y',\epsilon_k,n_k)\subset\end{equation}
$$\bigcup_{(i_{k-1},j_{k-1})\in [1,(10k)^{100}]^2} \bigcup_{u,v\in\Lambda_k}\bigcap_{(w,z)\in\{(i_{k-1},u),(j_{k-1},v)\}}
G_{-\tau_{wn_{k-1}}(\omega)}D\left(\sigma^{wn_{k-1}}\omega, G_{z} y', 
\delta_{k-1},
n_{k-1}\right).
$$

Fix $u,v$ and $(i,j)=(i_{k-1},j_{k-1})$. Then by invariance of the measure, 
$$
\nu\Big(G_{-\tau_{in_{k-1}}(\omega)}D(\sigma^{in_{k-1}}\omega, G_u y',  \delta_{k-1},n_{k-1})
\cap G_{-\tau_{jn_{k-1}}(\omega)}D(\sigma^{jn_{k-1}}\omega, G_{v}y',  \delta_{k-1},n_{k-1})\Big)=
$$
\begin{equation}\label{eq:numes}
\nu\Big(G_{\tau_{jn_{k-1}}(\omega)-\tau_{in_{k-1}}(\omega)}D(\sigma^{in_{k-1}}\omega, G_u y',  \delta_{k-1},n_{k-1})\cap D(\sigma^{jn_{k-1}}\omega, G_{v} y',  \delta_{k-1},n_{k-1})\Big).
\end{equation}
Since $i,j$ are $n_k$ good and $|i-j|\geq (10k)^{95}$,  it follows by \eqref{eq:zz1} that 
$$
\|\tau_{jn_{k-1}}(\omega)-\tau_{in_{k-1}}(\omega)\|\geq k^{25}n_{k-1}^{1/2}.
$$
Moreover, since $i,j$ are $n_k$ good, by \eqref{eq:zz2}, for $w\in\{i,j\}$,
$$
\sup_{r<n_{k-1}}\|\tau_r(\sigma^{wn_{k-1}})\|\leq 2k^{20}n_{k-1}^{1/2}.
$$
Therefore, by Lemma \ref{lem:induc} (with $\omega_w=\sigma^{wn_{k-1}}$), it follows that \eqref{eq:numes} is bounded 
from above by 
\begin{equation}\label{eq:abt}
C_\#\prod_{w\in\{i,j\}}\nu(D(\sigma^{wn_{k-1}}\omega, G_u y',  \delta_{k-1}+2^{-n_{k-1}^{1/2}},n_{k-1})).
\end{equation}
Moreover, since $i,j$ are good, $\sigma^{wn_{k-1}}(\omega)\in \Omega_{k-1}$.  
Also by \eqref{N-K}, $n_k\leq (1+1/100)\cdot 2^{n_k^{1/2}}$. Since $\inf \epsilon_k>0$ and $n_k$ grows exponentially, using \eqref{N-K} again, we have 
$$
\delta_{k-1}+2^{-n_{k-1}^{1/2}}=\Big(1-\frac{1}{100k^4}\Big)\epsilon_{k-1}+\frac{1}{n_k}+2^{-n_{k-1}^{1/2}}\leq \epsilon_{k-1}
$$
Using this, we obtain that \eqref{eq:abt} is bounded by
$\DS
C_\#(a_{k-1}(\epsilon_{k-1}))^2.
$
Using \eqref{eq:tpr2} and summing over all $u,u'\in\Lambda_k$ and $(i_{k-1},j_{k-1})\in[1,(10k)^{100}]^2$ (using that $k^{200}\leq n_k$), we have 
$$
a_k(\epsilon_k)\leq C_\#\cdot [100d(\sup |\tau|)n_k^2]^d\cdot (10k)^{200}\cdot a_{k-1}(\epsilon_{k-1})^{2}\leq$$$$ \Big(10^{200}\cdot C_\#\cdot (100d(\sup |\tau|))^d\Big)\cdot n_k^{2d+1} a_{k-1}(\epsilon_{k-1})^{2}.
$$
This by Lemma \ref{lem:a1} and Lemma  \ref{lem:aste} (with 
$C_2=10^{200}\cdot C_\#\cdot (100d(\sup |\tau|))^d$
and with $b_k = a_k(\epsilon_k)$) 
implies that $a_k(\epsilon_k)\to 0$ which finishes the proof.
\end{proof}

It remains to prove Lemma \ref{lem:tpr}.

\begin{proof}[Proof of Lemma \ref{lem:tpr}]
We consider the intervals $[rn_{k-1},(r+1)n_{k-1})$. 
Since $y\in D(\omega,y',\epsilon_k,n_k)$, it follows from the definition of $\{\epsilon_k\}$ that for at least $(10k)^{98}$ of $r<(10 k)^{100}$, the points 
\begin{equation}\label{eq:rn}
F^{rn_{k-1}}(\omega,y)\text{ and } F^{rn_{k-1}}(\omega',y') \text{ are }
\left(\left(1-\frac{1}{100k^4}\right)\epsilon_{k-1}, n_{k-1}\right){\text{-close}}.\end{equation}
Otherwise
 the cardinality of $i\leq n_k$ such that  $d\Big(F^i(\omega,y),F^{i}(\omega',y')\Big)<\epsilon_k$  
 would be bounded above by
$$
(10k)^{98}n_{k-1}+((10k)^{100}-(10k)^{98})n_{k-1}
\left(1-\left(1-\frac{1}{100k^4}\right)\epsilon_{k-1}\right)<
$$
$$
(10k)^{100}n_{k-1}\left(1-\left(1-\frac{1}{50k^2}\right)\epsilon_{k-1}\right)=n_k(1-\epsilon_k).
$$
This however contradicts the fact that $(\omega,y)$ and $(\omega',y')$ are $(\epsilon_k,n_k)$-close.  So there exists at least $(10k)^{196}$ pairs $(i,j)\in [0,(10k)^{100}]^2$ which satisfy \eqref{eq:rn}. Note that 
$$
\#\{(i,j)\in[0,(10k)^{100}]^2\;:\;|i-j|<(10k)^{95}\}\leq (10k)^{100+95}.
$$
Therefore 
$$
\#\{(i,j)\in[0,(10k)^{100}]^2\;:\;(i,j)\text{ satisfies }\eqref{eq:rn}\text{ and }|i-j|\geq (10k)^{95}\}\geq (10k)^{196}-(10k)^{195}.
$$
Moreover, since $\omega\in \Omega_k$, the cardinality of $n_k$--good 
pairs $(i,j)$ (see Definition \ref{def:toas}) is  at least $(10k)^{200}-5(10k)^{195}$.
Since $(10k)^{196}-(10k)^{195}>5(10k)^{195}$, it follows that there exists $(i,j)$ such that \eqref{eq:rn} holds for $r=i$ and $r=j$, and $(i,j)$ is $n_k$-good. This means that for $r=i,j$,
\begin{equation}\label{eq:zzz}
(\sigma^{rn_{k-1}}\omega',G_{\tau_{rn_{k-1}}(\omega')}y')\text{ and } (\sigma^{rn_{k-1}}\omega,
G_{\tau_{rn_{k-1}}(\omega)}y)
\end{equation}
are $\left(\left(1-\frac{1}{100k^4}\right)\epsilon_{k-1}, n_{k-1}\right)$-close.
Hence we find that for some $\|u_k\|\leq (\sup |\tau|) n_k$,  
 $$
G_{\tau_{in_k-1}(\omega)}y\in D(\sigma^{in_{k-1}}\omega, G_{u_i} y', (1-1/(100k^4))\epsilon_{k-1},n_{k-1}),
$$
and the same holds for $j$ with some $v_k$. This finishes the proof.
\end{proof}

\part{Technical lemmas}

\section{Ergodic sums of intermediate smoothness for toral translations}
\label{ScToral}

\begin{proof}[Proof of Proposition \ref{prop:dioph}]
We start with property $D1$, which is much simpler. Note that if $\DS \phi(x)=\sum_{k\neq 0} a_k e^{2\pi i \langle k, x\rangle} $ then 
$$ \phi_N(x)=\sum_{k\neq 0}
a_k e^{2\pi i \langle k, x\rangle} 
\frac{1-e^{2\pi i N \langle k, \alpha\rangle}}
{1-e^{2\pi i\langle k, \alpha\rangle}}. $$
Therefore
\begin{equation}
\label{ErgSumEigen}
 \|\phi_N\|_{2}^2 =\sum_{k\neq 0} |a_k|^2|A_k(N)|^2
\end{equation} 
where $A_k(N)=\frac{1-e^{2\pi i N \langle k, \alpha\rangle}}
{1-e^{2\pi i\langle k, \alpha\rangle}}$. A simple calculation gives

\begin{equation}
\label{ErgSumHarm}
|A_k(N)|= \left|\frac{1-e^{2\pi i N \langle k, \alpha\rangle}}
{1-e^{2\pi i \langle k, \alpha\rangle}}
\right|= \frac{|\sin(\pi N \langle k, \alpha\rangle)|}{|\sin(\pi  \langle k, \alpha\rangle)|}.
\end{equation}

{\bf Property D1} is a direct consequence of the following:
\begin{lemma}
If $\alpha\in \mathbb{D}(\bkappa)$,  $r<\bkappa$ and $\phi\in \bbH^r$ then
$\|\phi_N\|_2\leq C N^{1-(r/\bkappa)}. $
\end{lemma}
\begin{proof}
Using the estimate $\DS |A_k(N)|^2\leq C \min\left\{\langle k, \alpha \rangle^{-2}, N^2\right\}$, we get
$$ \|\phi_N\|_2^2\leq C \sum_{|k|\leq N^{1/\bkappa}} |k|^{2\bkappa} |a_k|^2+
\sum_{|k|\geq N^{1/\bkappa}} N^2 |a_k|^2=I+\RmII$$
where
$$ I\leq \sum_{|k|\leq N^{1/\bkappa}} (|k|^{2r} |a_k|^2) N^{2(\bkappa-r)}\leq 
C\|\phi\|_{\bbH^r}^2 (N^{1-(r/\bkappa)})^2, $$
and
$\DS  \RmII \leq \sum_{|k|\geq N^{1/\bkappa}} (|k|^{2r} |a_k|^2) (N^{1-(r/\bkappa)})^2\leq 
C \|\phi\|_{\bbH^r}^2(N^{1-(r/\bkappa)})^2 .$
\end{proof}

To establish property {\bf D2}
we start with the following lemma:
\begin{lemma}\label{lem:lar-1} There exists $R_\bm>0$ such that for every $N\in \mathbb{N}$ there exists $k_N\in  \integers^{\bm}$ satisfying:
$$
|\langle k_N,\alpha\rangle|<\frac{1}{4N}, \;\; |k_N|\leq R_\bm N^{1/\bm}.
$$
\end{lemma}
\begin{proof}

For $N\in \mathbb{N}$, consider the lattice
$$\cL(\alpha, N)= \left(\begin{array}{cccc}
 N^{-1/\bm} &  \dots & 0  & 0 \\
\dots & \dots & \dots & \dots \\
0 & \dots & N^{-1/\bm} &  0 \\
0 & \dots & 0 & N \end{array} \right) 
\left(\begin{array}{cccc}
 1 &  \dots & 0  & 0 \\
\dots & \dots & \dots & \dots \\
0 & \dots & 1 & 0 \\
\alpha_1 & \dots &  \alpha_\bm & 1 \end{array} \right) \mathbb{Z}^{\bm+1}\subset \R^{\bm+1}
$$
The points in this lattice are of the form 
$$e=(x,z)\in \reals^\bm \times \reals \text{ where } x=\frac{k}{N^{1/\bm }}, \;\; z=N\cdot(\langle k,
\alpha \rangle+m) \text{ and } (k,m)\in \integers^\bm\times \integers.$$
Let $R_\bm$ be such that a ball $\cB$ of radius $R_\bm$ in $\reals^\bm$ has volume $2^{\bm+3}.$ 
By Minkowski Theorem $\cL(\alpha, N)$ contains a non-zero vector $(x,z)$ in $\cB\times [-1/4,1/4].$ 
This finishes the proof.
\end{proof}

The above lemma has the following immediate consequence: 
\begin{lemma}\label{lem:lar1}There exists $c>0$ such that for every 
$l\in \mathbb{N}$ and every $N\in [2^l,2^{l+1}]$, we have
$$
\frac{|A_{k_{2^l}}(N)|}{|k_{2^l}|^r}\geq c\cdot N^{1-r/\bm}.
$$
\end{lemma}
\begin{proof} 
By the bound on $k_{2^l}$ from Lemma \ref{lem:lar-1} it suffices to show that
$$
|A_{k_{2^l}}(N)|\geq c'\cdot N.
$$
Note that  by Lemma \ref{lem:lar-1}, $|N\langle k_{2^l},\alpha\rangle|<1/2$. Now using
the estimate $C^{-1}<\frac{\sin z}{z}<C$ for
$z=N\langle k_{2^l},\alpha\rangle$ and $z=\langle k_{2^l},\alpha\rangle$
in \eqref{ErgSumHarm}, we obtain the result.
\end{proof}

Let $(k_{2^l})_{l\in \mathbb{N}}$ be the sequence from the above lemma. For a real sequence $\{a_l\}_{l\in \mathbb{N}}\subset [-1,1]$, let $\tau(a_l):\mathbb{T}^\bm\to \mathbb{C}$ be given by
\begin{equation}\label{eq:tau}
(\tau(a_{l}))(x)=\sum_{l> 0} \frac{ a_le^{2\pi i\langle k_{2^l},x\rangle}}{|k_{2^l}|^r\; l^2}.
\end{equation}

For $d\in \mathbb{N}$ let $\tau(a_l^{(1)},...,a_l^{(d)}): \mathbb{T}^\bm\to \mathbb{C}^{d}$
be defined by $(\tau(x))_j = (\tau(a_l^{(j)}))(x)$.
Let $\{a^{(j)}_l\}$ be i.i.d. random variables uniformly distributed on the unit cube in $\reals^d$ and the corresponding probability measure is denoted by $\Prob_{\bra}$.

\begin{lemma}
\label{LmUnlikelySmall}For every $\varepsilon >0$
there exists $C>0$ such that for every $x\in \mathbb{T}^\bm$ and every $N\in \mathbb{N}$,
$$ \Prob_{\bar{a}} \left(\|\tau_N(x)\|\leq N^{\eps} \right)<\left(\frac{C}{N^{1-r/\bm-2\eps}}\right)^{d} .$$
\end{lemma}

\begin{proof}
Since for a fixed $x$ different components of $\tau$ are independent, it suffices to consider the
case $d=1.$ In this case,  $\tau$ is given by \eqref{eq:tau}. Let $l$ be such that $N\in [2^l,2^{l+1}]$. We now fix all the $a_j$ for $j\neq l$. Then, since $N,x$ and all frequencies $2^j$ except $2^l$
are fixed, we can write (with some $\fc \in \mathbb{C}$ depending
on $a_j$, $j \neq l$ and $N$), 
\begin{equation}\label{eq:new}
\tau_N(x)=\fc+\frac{a_{l}A_{k_{2^l}}(N)e^{2\pi i\langle k_{2^l},x\rangle}}{|k_{2^l}|^r l^2}
\end{equation}
Let $M=(M_1,M_2):=\frac{1}{|k_{2^l}|^rl^2}
\left( \Re\Big(A_{k_{2^l}}(N)e^{2\pi i\langle k_{2^l},x\rangle}\Big),
\Im\Big(A_{k_{2^l}}(N)e^{2\pi i\langle k_{2^l},x\rangle}\Big) \right)$.\\ By Lemma~\ref{lem:lar1},
$$
|M|=\frac{|A_{k_{2^l}}(N)|}{|k_{2^l}|^{ r}l^2}\geq c\cdot N^{1-r/\bm-\epsilon}.
$$
Let us WLOG assume that $|M_1|\geq c/2\cdot N^{1-r/\bm-\epsilon}$
(if $|M_2|\geq c/2\cdot N^{1-r/\bm-\epsilon}$ the proof is analogous). 
It then follows that the measure of $z\in [-1,1]$ for which
$\DS
|M_1\cdot z-\Re(\fc)|<N^{\epsilon},
$
is bounded above by  $\DS \frac{2}{cN^{1-r/\bm -2\epsilon}}.$ 
Since $a_l$ is uniformly distributed on $[-1,1]$,\eqref{eq:new} finishes the proof. 
\end{proof}
Now we are ready to define the map $\tau$ and hence also finish the proof of {\bf D2}.

Take  $d\in \mathbb{N}$ such that $d(1-r/\bm-2\epsilon)>20$. 
Summing the estimates of Lemma~\ref{LmUnlikelySmall} over $N$,
we obtain that for some $C'>0$ and every fixed $x\in \mathbb{T}^\bm$, 
$$ 
\Prob_{\bar{a}} \left(\{\text{ there exists } N\geq n\;:\; \|\tau_N(x)\|\leq N^{\eps} \}\right)<\frac{C'}{n^{19}}.
$$
It follows by Fubini's theorem that 
$$ 
(\Prob_{\bar{a}}\times \mu)\Big(\{(a,x)\;:\; \text{ for all } N\geq n\;:\; \|\tau_N(x)\|\geq N^{\eps}\}\Big)\geq 1-\frac{C'}{n^{19}}.
$$
Using Fubini's theorem again, we get that there exists $\fA_n$ with $\Prob(\fA_n)\geq 1-\frac{C'}{n^8}$, such that for every $\bar{a}\in \fA_n$, 
$$
\mu(\{x\;:\; \text{ for all } N\geq n\;:\; \| (\tau(\bar{a}))_N(x)\|\geq N^{\eps}\})\geq 1-\frac{C'}{n^{7}}.
$$
 It is then enough to take $\DS \bar{a}\in \bigcap_{n\geq N_0} \fA_n$ for any fixed $N_0$
(notice that $\DS \bigcap_{n\geq N_0} {\fA_n}$ is non-empty if $N_0$ is large enough). Then the corresponding $\tau(\bar{a}):\mathbb{T}^\bm\to 
\mathbb C^d = \mathbb{R}^{2d}$ satisfies {\bf D2}
(with $2d$ instead of $d$). This finishes the proof of the proposition.
\end{proof}

\section{Ergodic integrals of flows on $\Tor^2$}
\label{ScSurfaceFlows}
Here we prove Proposition \ref{prop:deverg}.
We will study the flow $\varphi_t$ via its special representation.
We first prove some results on deviation of ergodic averages for functions with logarithmic singularities (either symmetric or asymmetric) and with power singularities. \\

For $N\in \mathbb{N}$, let $\DS \theta_{\min,N}:=\min_{j<N}\|\theta+j\alpha\|$, where $\theta \in \mathbb T$
and $\| z \| = \min \{ z, 1-z \}$.
 In the lemmas below we want to cover the cases
of logarithmic and power singularities simultaneously. 
 For roof functions with
 logarithmic singularities one can get much better bounds (with deviations
 being a power of $\log$) but we do not pursue the optimal bounds here since
 the bounds of the present section are sufficient for our purposes. 
 Let $J\in C^2(\mathbb{T}\setminus \{0\})$ be any function satisfying 
\begin{equation}\label{eq:asy}
\lim_{\theta\to 0^+}\frac{J(\theta)}{\theta^{-\gamma}}=P\text{ and }
\lim_{\theta\to 1^-}\frac{ J(\theta)}{(1-\theta)^{-\gamma}}=Q,
\end{equation}
for some constants $P,Q$. Notice that by 
l'Hosptial's rule it follows that any $f$ as in \eqref{eq:asy2} satisfies \eqref{eq:asy} 
(with $P=Q=0$ if $f$ has logarithmic singularities). Recall that $\gamma\leq 2/5$. 

In what follows, let $(a_n)$ denote the continued fraction expansion and $(q_n)$ denote the sequence of denominators of $\alpha$.

\begin{lemma}\label{lem:koksi}For every $x\in\mathbb{T}$ and every $n\in \mathbb{N}$,
$$
|J_{q_n}(\theta)-q_n\int_\mathbb{T} J(\vartheta) d\vartheta|=
{\rm O}\left(\theta_{\min,q_n}^{-\gamma}\right)
$$
\end{lemma}
\begin{proof} Let $\bar{J}(\theta)=(1-\chi_{[-\frac{1}{10q_n},\frac{1}{10q_n}]}(\theta))\cdot J(\theta)$. 
Then $\bar{J}$ is of bounded variation. Since $\Big|\{\theta+j\alpha\}_{j<q_n}\cap [-\frac{1}{10q_n},\frac{1}{10q_n}]\Big|\leq 1$, it follows that 
$$
|\bar{J}_{q_n}(\theta)-J_{q_n}(\theta)|= {\rm O}\left(\theta_{\min,q_n}^{-\gamma}\right),
$$
by the definition of $\theta_{min,q_n}$. By the Denjoy-Koksma inequality,
$$
|\bar{J}_{q_n}(\theta)-q_n\int_\mathbb{T} \bar{J}(\vartheta) d\vartheta|\leq {\rm Var}(\bar{J})={\rm O}(q_n^{\gamma}).
$$
Moreover, since $\DS \left|\{\theta+j\alpha\}_{j<q_n}\bigcap 
\left[-\frac{10}{q_n},\frac{10}{q_n}\right]\right|\geq 1$ it follows that 
$\DS \theta_{min,q_n}\leq \frac{10}{q_n}$, and so
$q_n^{\gamma} = {\rm O}\left(\theta_{\min,q_n}^{-\gamma}\right)$.
It remains to notice that 
$$
\left|\int_\mathbb{T} \bar{J}d\vartheta-\int_\mathbb{T}  Jd\vartheta\right|=
\int_{0}^{\frac{1}{10q_n}}Jd\vartheta+\int_{1-\frac{1}{10q_n}}^{1}Jd\vartheta={\rm O}
\left( q_n^\gamma/q_n\right),
$$
by the definition of $\bar{J}$. Since $\gamma<\frac{1}{2},$ the result follows.
\end{proof}

\begin{lemma}\label{lem:1b}
Assume that $\alpha$ is such that 
 $\DS \sup_{n\in \mathbb{N}} \frac{q_{n+1}}{q_n^{1+\zeta}}\leq C$
 for some $\zeta,C>0$.
 Then for every $N\in \mathbb{N}$
$$
\left|J_N(\theta)-N\int_\mathbb{T} J(\vartheta) d\vartheta\right|={\rm O}\Big(N^{\zeta}\log N\cdot \theta_{\min,N}^{-\gamma}\Big).
$$
\end{lemma}
\begin{proof} Let $\DS N=\sum_{k\leq M}b_kq_k$, with $b_k\leq a_k$, $b_M\neq 0$, $M={\rm O}(\log N)$ be the Ostrovski expansion of $N$. For every point $\bar{\theta}=\theta+j\alpha$, $j<N$ 
with $j+q_k<N$, we have that $\bar{\theta}_{\min,q_k}\geq \theta_{\min,N}$. Hence for each such point  Lemma \ref{lem:koksi} gives
$$
|J_{q_k}(\bar{\theta})-q_k\int_\mathbb{T} J(\vartheta) d\vartheta|=
{\rm O}\left(\theta_{\min,N}^{-\gamma}\right).
$$
Using cocycle identity, we write 
$\DS J_N(\theta)=\sum_{k\leq M}\sum_{j<b_k}J_{q_k}(\theta_{j,k})$, for some points 
$\bar{\theta}=\theta_{i,k}$ satisfying the above inequality for $q_k$. Then 
$$
\left|J_N(\theta)-N\int_\mathbb{T} J(\vartheta) d\vartheta\right|=
{\rm O}\left(M\cdot \sup_{k}b_k\cdot \theta_{\min,q_n}^{-\gamma}\right)=
{\rm O}\left(\log N\cdot N^{\zeta} 
\theta_{\min,N}^{-\gamma} \right),
$$
 where we use that $M={\rm O}(\log N)$ and 
 $$
 \sup_{k}b_k\leq \sup_{k}a_k={\rm O}(q_k^{\zeta})={\rm O}(N^\zeta).
 $$
This finishes the proof.
\end{proof}

Let $\alpha$ satisfy $q_{n+1}\leq Cq_n^{1+\zeta}$, for $0<\zeta<1/1000$. The set of such $\alpha$ has full measure by Khinchine's theorem.  

Let $c=\inf_{\mathbb{T}} f>0$. For $T>0$, we say that  $\theta\in \mathbb{T}$ is {\em $T$-good} if  
the orbit $\{\theta+j\alpha\}_{j\leq \frac{T}{c}}$ does not visit the interval 
$\DS \left[-\frac{1}{T^{1+1/100}},\frac{1}{T^{1+1/100}}\right]$. We have the following
\begin{lemma}\label{x-tr}Let $T_t^f$ be a special flow with $f$ satisfying \eqref{eq:asy2}.
$$
W(T):=\{(\theta,s)\;:\; \theta\text{ is }T\text{-good} \}.
$$
Then $\mu(W(T))=1-o(1)$ as $T\to \infty $.
\end{lemma}
\begin{proof} For an interval $I$, let $I^f:=\{(\theta, s): s<f(\theta), \theta\in I\}$. Note that 
$$
 (W(T))^c=\bigcup_{j\leq \frac{T}{c}}I_j^f,
$$
where $\DS I_j=\left[-j\alpha-\frac{1}{T^{1+1/100}},-j\alpha+\frac{1}{T^{1+1/100}}\right]$. 
Moreover, by the diophantine assumptions on $\alpha$, all the intervals $I_j$ are pairwise disjoint. Therefore, for $j\neq 0$, 
\begin{equation}
\label{IntMes1}
\sup_{\theta\in I_j} f(\theta)\leq C\cdot T^{(1+1/100)\gamma}.
\end{equation}
Hence $\DS \mu\left(\bigcup_{0\neq j\leq \frac{T}{c}}I_j^f\right)\leq C T^{(1+1/100)\gamma}$. Moreover, since $f$ satisfies \eqref{eq:asy2}
\begin{equation}
\label{IntMes2}
\mu(I_0^f)=o(1),\quad \text{as}\quad T\to\infty.
\end{equation}
Combining \eqref{IntMes1} and \eqref{IntMes2} gives the result.
\end{proof}

Using the three lemmas above we can prove Proposition \ref{prop:deverg}.

\begin{proof} Let $\alpha$ satisfy $q_{n+1}\leq Cq_n^{1+\zeta}$, for $0<\zeta<1/1000$. We will show that there exists $C>0$ such that for every $T$, and every $(\theta,s)\in W(T)$, we have 
$$
|\bar{H}_T(\theta,s)-T\mu(\bar{H})|\leq CT^{1/2-1/1000}.
$$
This by Lemma \ref{x-tr} will finish the proof of the proposition. Notice that for $(\theta,s)\in W(T)$, 
we have in particular that 
$s<f(\theta)\leq C T^{(1+1/100)\gamma} \leq CT^{1/2-1/1000}$
$$
|\bar{H}_T(\theta,s)-\bar{H}_T(\theta,0)|<
\|\bar{H}\|_1\;s\leq C'\|\bar{H}\|T^{1/2-1/1000}.
$$
Therefore, it is enough to show that if $(\theta,0)\in W(T)$, then 
\begin{equation}\label{eq:dev2}
|\bar{H}_T(\theta,0)-T\mu(\bar{H})|\leq C''T^{1/2-1/1000}.
\end{equation}
for some constant $C''>0$. Note that
\begin{equation}
\label{NumCross}
cN(\theta,0,T)\leq |f_{N(\theta,0,T)}(\theta)|\leq T
\end{equation}
and so $\DS \|\theta+N(\theta,0,T)\alpha\|\geq 
\min_{j\leq \frac{T}{c}}\|\theta+j\alpha\|\geq T^{-1-1/100}$. In particular 
$$f(\theta+N(\theta,0,T)\alpha)\leq C'''T^{(1+1/100)\gamma}. $$
So
$$\int_0^T\bar{H}(\varphi_t(\theta,0))dt-T\mu(\bar{H})=
$$$${\rm O}\left(T^{(1+1/100)\gamma}\right)+\left(\int_{0}^{N(\theta,0,T)}
\bar{H}(\varphi_t(\theta,0))dt-N(\theta,0,T)\mu(\bar{H})\right)+(T-N(\theta,0,T))\mu(\bar{H}).$$
Since $\gamma\leq 2/5$, it is enough to bound the second and last term above. It is therefore enough to prove the following: for every $(\theta,0)\in W(T)$,
 \begin{equation}\label{eq:devbo2}
 |T-N(\theta,0,T)|={\rm O}\left(T^{1/2-1/1000}\right),
\end{equation}
and 
 \begin{equation}\label{eq:devbo3}
\Big|\int_{0}^{N(\theta,0,T)}\bar{H}(\varphi_t(\theta,0))dt-N(\theta,0,T)\mu(\bar{H})\Big|={\rm O}\left(T^{1/2-1/1000}\right).
\end{equation}
To prove \eqref{eq:devbo2} 
note that for $(\theta,0)\in W(T)$
$$
f_{N(\theta,0,T)}(\theta)\leq T\leq f_{N(\theta,0,T+1)}(\theta)
\leq f_{N(\theta,0,T)}(\theta)+C'''T^{(1+1/100)\gamma}.
$$
Hence up to an additional negligible error of size $T^{(1+1/100)\gamma}$, it is enough to control
$$
|f_{N(\theta,0,T)}(\theta)-N(\theta,0,T)|.
$$
 By \eqref{NumCross} and our assumption on $\theta$ it follows that $\theta_{\min,N(\theta,0,T)}\geq T^{-1-1/100}$. 
So Lemma \ref{lem:1b}, the above upper bound on $N(\theta,0,T)$ and the fact that $\int_{\mathbb{T}} f dLeb = 1$ imply that 
$$
|f_{N(\theta,0,T)}(\theta)-N(\theta,0,T)|\leq 
{\rm O}\left(T^{\zeta+(1+1/100)\gamma}\log T\right).
$$
Since $\zeta+(1+1/100)\gamma\leq 1/1000+(1+1/100)2/5\leq 1/2-1/1000$, 
\eqref{eq:devbo2}  follows. \medskip

To prove \eqref{eq:devbo3} we can WLOG assume that $\mu(\bar{H})=0$. Note that 
$$
\int_{0}^{N(\theta,0,T)}\bar{H}(\varphi_t(\theta,0))dt=\sum_{i=0}^{N(\theta,0,T)-1}\int_{0}^{f(\theta+i\alpha)}\bar{H}(\theta+i\alpha,s)ds=\sum_{i=0}^{N(\theta,0,T)-1}F(\theta+i\alpha)
$$
where $\DS F(\theta)=\int_0^{f(\theta)}\bar{H}(\theta,s)ds$. 
 Moreover, $Leb(F)=\mu(\bar{H})=0$ and $F$ is smooth except at $0$. 
 Since $f$ satisfies \eqref{eq:asy}  and $\brH\in \cC^3$, it follows that 
 $$
\lim_{\theta\to 0^+}\frac{F(\theta)}{\theta^{-\gamma}}=P'\text{ and } 
\lim_{\theta\to1^-}\frac{F(\theta)}{(1- \theta)^{-\gamma}}=Q'
 $$
 where $P'=P\bp(\brH)$, $Q'=Q\bp(\brH)$. Thus $F(\cdot)$ also satisfies the assumptions \eqref{eq:asy}. 
 So by Lemma \ref{lem:1b}, the fact that $(\theta,0)\in W(T)$ and the bound 
 $N(\theta,0,T)\leq \frac{T}{c}$,
 $$
 \left|\sum_{i=0}^{N(\theta,0,T)-1}F(\theta+i\alpha)\right|=
 {\rm O}\left(T^{\zeta+(1+1/100)\gamma}\log T\right)=
 {\rm O}\left(T^{1/2-1/1000}\right).
 $$
 This finishes the proof of \eqref{eq:devbo3} and 
completes the proof of the proposition.
\end{proof}

\section{Ergodic sums over hyperbolic maps and subshifts of finite type}

\subsection{CLT for higher rank Kalikow systems. Proof of Theorem \ref{ThAnosovCLT}(ii)}
\label{SSAnCLT}
As in Section \ref{ScStandard} we define $\fm_N$ by \eqref{eq:disc} and check the conditions of 
Proposition \ref{PrBG}. (a) is evident. Also, by the local limit theorem we get
$\DS \mu(\sigma_{0, k})=O\left(k^{-d/2}\right)$ which implies equation \eqref{Eq1Cor}
with $\beta=d/2$ which in case $d\geq 3$ is sufficient
to prove (c) in the same way as in Section \ref{ScStandard},
see footnote \ref{FtBeta}.

To prove property (b),
let $\ell(x, t, N)=\Card\{n \leq N:  |\tau_n(x)-t|\leq 1\}$.
Using multiple LLT we get that
for each $p$, there is a constant $C_p$ such that
for each $t\in \reals^d$ for each $n$
$$ \mu\left(\ell^p(\cdot, t,  n)\right)\leq C_p $$
(see e.g. \cite[Section 5]{DSV08}). Now the Markov inequality implies that for each $\eps, t, p$ we have
$$ \mu\left(x:  \ell(x, t, N)\geq N^{(1/5)-\eps}\right)
\leq \frac{C_p}{N^{[(1/5)-\eps]p}}. $$
It follows that
$$ \mu\left(x:  \exists t: \|t\|\leq K\ln N\text{ and }\ell(x, t, N)\geq N^{(1/5)-\eps}\right)\leq 
\frac{C_p (K\ln N)^d}{N^{[(1/5)-\eps]p}}. $$
Taking $p=6,$ $\eps=0.01$, we verify the conditions of Lemma \ref{LmQuarter}.

\subsection{Visits to cones} 
\label{SSCones}
\begin{proof}[Proof of Lemma \ref{lem:RW}]
We only prove 
the case of $\mathbb Z_+$, as the case of $\mathbb Z_-$ is similar.
Let
$$\hC =\{v\in \cC: \text{dist}(v, \partial \cC)\geq 1 \}. $$
Define $n_1 = 2$ $n_{k+1}=n_k^3$ and
$$A_k=\{\omega: 
\tau_{n_k} (\omega)\in \cC
\text{ and }
\| \tau_{n_k} (\omega) \|  > \sqrt{n_k} 
\}. $$
It suffices to show that infinitely many $A_k$ happen with probability 1.
Since $\phi$ only depends on the past,
$A_k$ is measurable with respect to $\cF_k$, the $\sigma$-algebra generated by 
$\omega_j$ with $j\leq n_k.$
Therefore by L\'evy's extension of the Borel-Cantelli Lemma (see e.g. \cite[\S 12.15]{W91}) it is enough to show that 
for almost all $\omega$
\begin{equation}
\label{GenBC}
 \sum_k \mu(A_{k+1}|\cF_k)=\infty. 
\end{equation} 
However by 
mixing central limit theorem, 
there is $\eps=\eps(\cC)$ such that
for any cylinder $\cD$ of length $n_k$
$$ \mu\left({\tau_{n_{k+1}-n_k} (\sigma^{n_k} \omega)}  \in \hC,\;
\| \tau_{n_{k+1}-n_k} (\sigma^{n_k} \omega) \| > 
\sqrt{n_{k+1} - n_k}
\Big| \omega\in \cD \right) \geq 
\eps. $$
Since $\| \tau_n \|_{\infty} \leq n \| \tau \|_{\infty}$ and 
$\sqrt{n_{k+1}}/n_k\to \infty$, we conclude from the last display that each term in \eqref{GenBC}
is greater than $\eps.$ This completes the proof.
\end{proof}

\subsection{Separation estimates for cocycles}

\label{ScSep}
\begin{proof}[Proof of Lemma \ref{lem:cl}]
({\bf m2}) follows from the fact that 
 there exists a constant $C_\tau$ such that
if $\omega'$ and $\omega''$ belong to the same cylinder of length $N$, then 
$$ |\tau_N(\omega')-\tau_N(\omega'')|\leq C_\tau.
$$
To prove ({\bf m1}) let
$$ N_A(\omega, k)\!=\!\#\left\{(i,j)
 \in [0,(10k)^{100}]\times [0,(10k)^{100}], i\neq j:\,
\frac
{\|\tau_{(j-i)n_{k-1}}(\sigma^{in_{k-1}}\omega)\|}{(|j-i|n_{k-1})^{1/2}}< k^{-20}\right\}.$$
Denote $m_{ij}=|i-j| n_{k-1}.$ Covering the ball with center at the origin and radius
$\DS \frac{\sqrt{m_{ij}}}{k^{20}}$ in $\R^d$ by unit cubes and applying the anticoncentration
inequality \cite[formula (A.4)]{DDKN} to each cube, we obtain that 
\begin{equation}
\label{AntiCCor}
\mu\left(\|\tau_{m_{ij}}(\omega)\|\leq \frac{ \sqrt{m_{ij}}}{k^{20}} \right)\leq C k^{-20 d}. 
\end{equation}
Since $\mu$ is shift invariant we conclude that 
$$\mu\left(\frac
{\|\tau_{m_{ij}}(\sigma^{in_{k-1}}\omega)\|}{m_{ij}^{1/2}}
< \frac{1}{k^{20}}\right)\leq C k^{-20 d}.$$
Summing over $i$ and $j$ we obtain
$$ \mu\left(N_A(\cdot, k)\right)\leq C (10k)^{200-20 d}. $$
Next, by the Markov inequality,
$$ \mu\left(\omega: N_A(\omega, k)\geq (10k)^{191}\right)\leq \frac{C}{k^{20d-9}}. $$
This shows that the measure of the complement of $A_k$ is small.
The estimate of measure of $B_k$ is similar except we replace \eqref{AntiCCor}
by
\begin{equation}
\label{MDMax}
 \mu\left(\max_{n\leq m} \|\tau_n(\omega)\|\geq k^{20}\sqrt{m} \right) \leq c_1 e^{-c_2 k^{40}}. 
\end{equation}
To prove \eqref{MDMax}  it is sufficient to consider the case $d=1$ since for higher dimensions we can consider 
each coordinate separately. Thus it suffices to show that
\begin{equation}
\label{MDMax1}
 \mu\left(\max_{n\leq m} \tau_n(\omega)\geq k^{20}\sqrt{m} \right) \leq c_1 e^{-c_2 k^{40}}
\end{equation}
(the bound on
$\DS \mu\left(
\min_{n\leq m} \tau_n(\omega)
\leq - k^{20}\sqrt{m} \right) $
is obtained by replacing $\tau$ by $-\tau.$).

To prove \eqref{MDMax1} with $d=1$ we use the reflection principle. Namely, \cite[formula (A.3)]{DDKN}
shows that for each $L$
\begin{equation}
\label{MDFix}
 \mu\left(|\tau_m(\omega)|\geq L \sqrt{m} \right) \leq \brc_1 e^{-\brc_2 L^2}. 
\end{equation}
Let
$$ D_m(k)=\left\{\omega: \exists n\leq m,  \tau_n(\omega)\geq k^{20}\sqrt{m} \right\}. $$
Note that $D_{m}(k)$ contains the LHS of \eqref{MDMax1} and that $D_m(k)$ is a disjoint union of the
cylinders of length at most $m$,
$\DS D_m=\bigcup_j \cC_j$
(to see this, take for each $\omega$ the smallest $n$ such that 
the last display holds and recall that
$\tau$ only depends on the past). 
Next, there exists $\ell=\ell(\tau)$ such that for each cylinder $\cC$ of length $n=n(\cC)$ 
and for each $m$,
$$ \mu\left(\{\tau_{m-n} \geq -\ell|\omega\in \sigma^{-n}\cC\}\right)\geq \frac{1}{2} . $$
If $m-n$ is large this follows from (mixing) Central Limit Theorem \cite{PP, Ea76} 
while the small $m-n$ could be handled by choosing $\ell$ large enough. 
Combining this with \eqref{MDFix}, we obtain
$$ \brc_1 e^{-\brc_2 k^{40}/4}\geq  \mu\left(\tau_m\geq \frac{k^{20} \sqrt{m}}{2}\right)\geq 
\sum_j \mu\left(\omega\in \cC_j,\; \tau_m\geq \frac{k^{20} \sqrt{m}}{2}\right)\geq 
$$
$$ \sum_j \mu(\cC_j) \mu\left(\tau_m\geq \frac{k^{20} \sqrt{m}}{2}\Big|\omega\in \cC_j\right)\geq 
\frac{1}{2} \sum_j \mu(\cC_j)=\frac{\mu(D_m)}{2}
$$
proving \eqref{MDMax1} and completing the proof of the lemma.
\end{proof}

{\bf Acknowledgements:} 
We thank Bassam Fayad and Jean-Paul Thouvenot for useful discussions.
D.\ D.\ was partially supported by the NSF grant
DMS-1956049,
A.\ K.\ was partially supported by the NSF grant DMS-1956310,
P.\ N.\ was partially supported by the NSF grants DMS-1800811 and DMS-1952876.

\end{document}